\documentclass[12pt]{amsart}
\usepackage[utf8]{inputenc}
\usepackage{amsfonts}
\usepackage[all]{xy}
\usepackage{graphicx}
\usepackage{amscd, amsmath, mathrsfs, amssymb, amsthm, amsxtra, bbding, epsfig, eucal, eufrak, graphicx, latexsym, mathrsfs}

\usepackage{tikz}

\newcommand{\rr}{\mathbb R}

\newcommand{\dd}{{\rm d}}

\newcommand{\re}{\Re e\,}
\newcommand{\im}{\Im m\,}

\def\leq{\leqslant}
\def\geq{\geqslant}
\def\le{\leqslant}
\def\ge{\geqslant}

\theoremstyle{plain}
\newtheorem{theorem}{Theorem}
\newtheorem{corollary}{Corollary}
\newtheorem{lemma}{Lemma}[section]

\newtheorem{proposition}{Proposition}[section]

\theoremstyle{definition}

\newtheorem{remark}{Remark}

\numberwithin{equation}{section}

\begin{document}

\title[On the distribution of values of symmetric power $L$-Functions at 1]
{Distribution of values of symmetric power $L$-Functions at the edge of the critical strip}
\author{Xuanxuan Xiao}
\address{Institut \'{E}lie Cartan de Lorraine,  Universit\'{e} de Lorraine, 54506 Vandoeuvre les Nancy, France}
\email{xuanxuan.xiao@univ-lorraine.fr}
\thanks{The author is supported by China Scholarship Council (CSC) and Graduate Innovation Foundation of Shandong University (GIFSDU)}


\begin{abstract}We study some problems on the distribution of values of symmetric power $L$-functions at $s=1$ 
in both aspects of level and weight:
bounds of these values, extreme values, Montgomery-Vaughan's conjecture and distribution functions.
Our results generalize and/or improve related results of Royer-Wu \cite{Royer_Wu}, Cogdell-Michel \cite{Cogdell_Michel},
Lau-Wu \cite{Lau_Wu, Lau_Wu2} and Liu-Royer-Wu \cite{Liu_Royer_Wu}.
\end{abstract}
\smallskip
\maketitle

\section{Introduction}

The values of $L$-functions at the edge of the critical strip contain interesting arithmetic information.
In the case of Riemann $\zeta$-function, 
it is well known that the prime number theorem is equivalent to the non-vanishing of $\zeta(1+\text{i}\tau)$ for $\tau\in \rr$.
The study on distribution of values of Dirichlet $L$-functions associated with real primitive characters $\chi_d$ at $s=1$ 
has a long and rich history.
We refer the reader to Granville and Soundararajan's excellent paper \cite{Granville_Soundararajan} for a detail historical description. 
In particular they \cite[Theorem 1]{Granville_Soundararajan} proved a deep conjecture of Montgomery and Vaughan 
concerning the distribution of values of $L(1, \chi_d)$ (see \cite[Conjecture 1]{Montgomery_Vaughan})
with the help of Graham-Ringrose's bounds
for short character sums with highly composite moduli \cite{MR1084186}.

In this paper we are interested in the distribution of values of the symmetric power $L$-functions at $s=1$
in the level-weight aspect.
Let us begin by presenting some standard notations in this field.
Let $k$ be a positive even integer and $N$ be a positive square free integer. 
Denote by $\mathcal{H}^*_k(N)$ the set of normalised newforms of level $N$ and of weight $k$.
We have
\begin{equation}\label{cardinalHkN*}
|\mathcal{H}^*_k(N)|=\frac{k-1}{12} \varphi(N)+O\big((kN)^{2/3}\big),
\end{equation}
where $\varphi(N)$ is the Euler function and the implied constant is absolute.
Denote by ${\mathbb H}$ the upper half complex plane and 
write the Fourier expansion of $f\in \mathcal{H}^*_k(N)$ at the cusp $\infty$ as 
\begin{equation*}
f(z)=\sum\limits_{n=1}^{\infty}\lambda_f(n)n^{(k-1)/2} {\rm e}^{2\pi {\rm i}nz} 
\qquad
(z\in {\mathbb H}),
\end{equation*}
where $\lambda_f(n)$ is the $n$-th normalized Fourier coefficient of $f$ satisfying the Hecke relation
\begin{equation}\label{Heck}
\lambda_f(m) \lambda_f(n) = \sum_{\substack{d\mid (m, n)\\ (d, N)=1}} \lambda_f\bigg(\frac{mn}{d^2}\bigg)
\end{equation}
for all integers $m, n\ge 1$.
According to Deligne, 
for $f\in \mathcal{H}^*_k(N)$ and any prime number $p$ there are $\varepsilon_f(p)=\pm 1$, $\alpha_f(p)$ and $\beta_f(p)$ 
such that 
\begin{equation}\label{eq 1}
\left\{
\begin{aligned}
&|\alpha_{f}(p)|=\alpha_f(p)\beta_f(p)=1  &&\mbox{if $p \nmid N$}
\\
&\alpha_{f}(p)=\varepsilon_{f}(p)/\sqrt{p}, \beta_f(p)=0 &&\mbox{if $p \mid N$}
\end{aligned}
\right.
\end{equation}
and
\begin{equation}\label{01}
\lambda_f(p^\nu) = \alpha_f(p)^{\nu} + \alpha_f(p)^{\nu-1}\beta_f(p) + \cdots + \beta_f(p)^{\nu}
\qquad 
(\nu \geq 0).
\end{equation}

The $m$-th symmetric power $L$-function attached to $f\in \mathcal{H}^*_k(N)$ is defined as
\begin{equation}\label{defLssymmf}
L(s, {\rm sym}^m f)
:= \prod_p \prod_{0\leq j \leq m} \big(1-\alpha_f(p)^{m-j}\beta_f(p)^j p^{-s}\big)^{-1}
\end{equation}
for $\sigma>1$, 
where and in the sequel we define implicitly real numbers $\sigma$ and $\tau$ by the relation $s=\sigma+\rm i\tau$.
According to \cite[Section 3.2.1]{Cogdell_Michel}, the gamma factors of $L(s,{\rm sym}^m f)$ are
\begin{equation}\label{GammaFactor}
L_\infty(s,{\rm sym}^mf)
:=
\left\{
\begin{aligned}
&\prod_{0\le \nu\le n} \Gamma_{\mathbb{C}}\big(s+(\nu+\tfrac{1}{2})(k-1)\big) &&\mbox{if $m=2n+1$},
\\
&\Gamma_{\mathbb{R}}\left(s+\delta_{2\nmid n}\right)\prod_{1\le \nu\le n} \Gamma_{\mathbb{C}}(s+\nu(k-1))  &&\mbox{if $m=2n$},
\end{aligned}
\right.
\end{equation}
where $\Gamma_{\mathbb{R}}(s):=\pi^{-s/2}\Gamma(s/2)$, $\Gamma_{\mathbb{C}}(s):=2(2\pi)^{-s}\Gamma(s)$ 
and $\delta_{2\nmid n}=1$ if $2\nmid n$ and $0$ if not.
For $1\le m\le 4$, the complete symmetric power $L$-function
\begin{equation*}
\Lambda(s,{\rm sym}^m f) := N^{ms/2}L_\infty(s,{\rm sym}^m f)L(s,{\rm sym}^m f)
\end{equation*}
is entire and satisfies the functional equation
\begin{equation*}
\Lambda(s,{\rm sym}^m f)=\varepsilon_{{\rm sym}^m f} \Lambda(1-s,{\rm sym}^m f)
\qquad
(s\in {\mathbb C}),
\end{equation*}
where $\varepsilon_{{\rm sym}^m f}=\pm 1$.

\subsection{Bounds of $L(1, {\rm sym}^m f)$ and its extreme values}\

\vskip 1mm

The distribution of values of symmetric $L$-functions at $s=1$ has received attention of many authors during the last twenty years
\cite{Hoffstein_Lockhart, Luo, Royer2, Royer1, Habsieger_Royer, Royer_Wu, Cogdell_Michel, Lau_Wu, Lau_Wu2, Liu_Royer_Wu}. 
Diverse methods or techniques have been developed and great progress achieved. 

When $f\in \mathcal{H}_k^*(N)$ and $m=1,2$, 
Hoffstein and Lockhart \cite{Hoffstein_Lockhart} proved that
\begin{equation}\label{HoffLock}
(\log (kN))^{-1}\ll L(1,{\rm sym}^m f)\ll \log (kN),
\end{equation}
where the implied constants are absolute.

Luo \cite{Luo} considered the case of Maass forms.
Let $\{f_j(z)\}$ be an orthonormal Hecke basis of $L_0^2(\Gamma\setminus\mathbb{H})$
and $\tfrac{1}{4}+t_j^2\;(t_j\ge 0)$ be the Laplacian eigenvalue of $f_j(z)$. 
He proved that 
\begin{equation}\label{MaassMoment} 
\lim_{T\to \infty} \frac{1}{|\{j  :  t_j\leqslant T\}|} \sum_{t_j\leqslant T} L(1, {\rm sym}^2f_j)^{r-1}
= M_{{\rm sym}^2}^{r}
\end{equation}
for all \textit{integers} $r\geqslant 1$,
where $M_{{\rm sym}^2}^{r}$ is a positive constant depending on $r$ and 
verifying $\log M_{{\rm sym}^2}^{r}\ll r\log_2r$
($\log_j$ denotes the $j$-fold iterated logarithm. See \eqref{defMsymmz} below for an explicit expression for $M_{{\rm sym}^2}^{r}$). 
As an immediate consequence of \eqref{MaassMoment}, he stated the following corollary:
\begin{equation}\label{DistributionFunction}
\lim_{T\to \infty} \frac{1}{|\{j  :  t_j\leqslant T\}|}
\sum_{\substack{t_j\le T\\ L(1, {\rm sym}^2f_j)\leqslant t_j}} 1
= F(t)
\end{equation}
at each point of continuity of a distribution function $F(t)$.

In \cite{Royer2}, Royer considered the holomorphic case.
Denote by $P^{-}(n)$ the least prime factor of $n$ with the convention $P^-(1)=\infty$.
He established the analogue of \eqref{MaassMoment} for holomorphic forms :
\begin{equation}\label{HolomorphicMoment}
\lim_{\substack{N\to \infty\\ P^-(N)\ge N^{\varepsilon}}}
\frac{1}{|\mathcal{H}_k^*(N)|}\sum_{f\in \mathcal{H}_k^*(N)} L(1,{\rm sym}^2f)^{\pm r}
= M_{{\rm sym}^2}^{\pm r}
\end{equation}
for all \textit{integers} $r\geqslant 1$ and any $\varepsilon>0$, and showed that 
$$
\log M_{{\rm sym}^2}^{r}=3r\log_2 r+O(r)
\quad
(r\to+\infty).
$$ 
Some interesting combinatorial interpretations on $M_{{\rm sym}^m}^{-r}$ and $M_{{\rm sym}^m}^{r}\;(m=1, 2)$
can be found in \cite{Royer1} and \cite{Habsieger_Royer}, respectively.
Further the authors of these papers showed, with the help of these combinatorial interpretations, that
\begin{align}
\log M_{{\rm sym}^1}^{-r}
& = 2r\log_2 r + 2(\gamma-2\log\zeta(2))r + O(r/\log r),
\label{-r,1}
\\\noalign{\vskip 1mm}
\log M_{{\rm sym}^2}^{-r}
& = r\log_2 r + (\gamma-2\log\zeta(2))r + O(r/\log r),
\label{-r,2}
\\\noalign{\vskip 1mm}
\log M_{{\rm sym}^m}^{r}
& = (m+1)r\log_2 r+(m+1)\gamma r + O(r/\log r)
\quad
(m=1, 2),
\label{r,1,2}
\end{align}
for $r\to\infty$, where $\gamma$ is the Euler constant.
From \eqref{HolomorphicMoment}, \eqref{-r,2} and \eqref{r,1,2} with $m=2$, we immediately deduce that the set
$$
\big\{L(1, {\rm sym}^2f), L(1, {\rm sym}^2f)^{-1} : f\in \mathcal{H}_k^*(N)\big\}
$$
is not bounded when $N\to\infty$ with $P^-(N)\ge N^{\varepsilon}$.

In order to give a quantitative version of this statement,  
Royer and Wu \cite{Royer_Wu} analysed dependencies in parameters $N$ and $r$ carefully. 
This analysis requires a radical change of techniques used in \cite{Royer2}. Let \begin{equation}\label{harmweight}
\omega_f := \frac{2\pi^2}{(k-1)\varphi(N) L(1,{\rm sym}^2 f)}
\end{equation}
be the harmonic weight which appears in Petersson trace formula.
They sharpened \eqref{HolomorphicMoment} as follows :
Let $k$ be a \textit{fixed} even integer. Then there is an absolute constant $c$ such that
\begin{equation}\label{AF_RW}
\sum_{f\in \mathcal{H}_k^*(N)} \omega_f L(1,{\rm sym}^2f)^{\pm r}
= M_{{\rm sym}^2}^{\pm r}\big\{1+O\big((\log_2 N)^{-1}\big)\big\}
+ O_k\big(N^{-1/13}\text{e}^{cr\sqrt{\log(3N)}+cr^2}\big)
\end{equation}
uniformly for all $r\in {\mathbb N}$ and $N\in {\mathbb N}$ with $P^-(N)\ge \log N$,
where the implied constant depends only on $k$.
From this it is easy to deduce that there is $f_{\pm}\in \mathcal{H}_k^*(N)$ such that
\begin{equation}\label{ExtremValueRW}
L(1,{\rm sym}^2f_{-})\ll_k (\log_2N)^{-1},
\qquad
L(1,{\rm sym}^2f_{+})\gg_k (\log_2N)^{3}
\end{equation}
for all $N\in {\mathbb N}$ with $P^-(N)\ge \log N$.
Further they also showed that 
\begin{equation}\label{CardementRW}
(\log_2N)^{-1}\ll_k L(1,{\rm sym}^2f)\ll_k (\log_2N)^{3}
\end{equation}
for all $N\in {\mathbb N}$ with $P^-(N)\ge \log N$ and $f\in \mathcal{H}_k^*(N)$ 
provided the Generalized Riemann Hypothesis (GRH) for $L(s, {\rm sym}^2f)$ holds.
Therefore \eqref{ExtremValueRW} is optimal with regard to the order of magnitude.
They also showed that the set
$$
\big\{L(1, {\rm sym}^2f), L(1, {\rm sym}^2f)^{-1} : f\in \mathcal{H}_k^*(N_j)\big\}
$$
is bounded when $j\to\infty$, where $p_j$ is the $j$-th prime and $N_j = p_1\cdots p_j$.
Therefore a condition of type $P^-(N)\ge \log N$ is indispensable.

In \cite{Cogdell_Michel}, Cogdell and Michel introduced a more conceptual approach.
By providing a natural probabilistic interpretation, 
they interpreted the complex moments for symmetric power $L$-functions 
by the expected value of an Euler product defined on the probability space :
\begin{equation}\label{defMsymmz0}
M_{{\rm sym}^m}^{z}
= \prod_p \frac{2}{\pi} \int_0^{\pi} \prod_{j=0}^{m} \big(1-\text{e}^{\text{i}(m-2j)\theta}p^{-1}\big)^z \sin^2\theta \dd \theta.
\end{equation}
This new method has two advantages: 
On the one hand, they can calculate the complex moments of $L(1, {\rm sym}^mf)$ for all integers $m\ge 1$
(unconditionally for $1\le m\le 4$ and under their hypothesis ${\rm sym}^m(N)$ for $m\ge 5$ :
For all $f\in \mathcal{H}_k^*(N)$, $L(s, {\rm sym}^mf)$ is automorphic.);
on the other hand, with the help of the formula \eqref{defMsymmz0},
they can rather easily evaluate $M_{{\rm sym}^m}^{r}$ for all real $r\to\infty$
(avoiding complicated combinatorial analyze in \cite{Royer1, Habsieger_Royer}).
Thanks to this new method, 
Codgell and Michel can generalize and improve Royer-Wu's \eqref{CardementRW} and \eqref{ExtremValueRW} as follows: 
Let $N$ be a \textit{prime} and $f\in \mathcal{H}_2^*(N)$.
Under GRH for $L(s,{\rm sym}^mf)$, one has
\begin{equation}\label{CardementCM}
\{1+o(1)\}(2B^-_m\log_2N)^{-A_m^-} \leqslant L(1,{\rm sym}^mf)\leqslant \{1+o(1)\}(2B^+_m\log_2N)^{A_m^+}
\end{equation}
for $N\to\infty$, 
where $A^{\pm}_m$ and $B^{\pm}_m$ are positive constants defined as in \eqref{defAmpmBmpm} below. 
We have
\begin{equation}\label{defAmBm}
\left\{
\begin{aligned}[l]
& A_m^{+}=m+1,        &     &\qquad B_m^+={\rm e}^\gamma                     &    &\qquad (m \in \mathbb{N}),
\\
& A_m^{-}=m+1,        &     &\qquad B_m^-={\rm e}^\gamma\zeta(2)^{-1}   &    &\qquad (2 \nmid m),
\\
& A_2^-=1,                 &     &\qquad B_2^-={\rm e}^\gamma\zeta(2)^{-2},   &    &
\\
& A_4^-=\tfrac{5}{4},  &     &\qquad B_4^-={\rm e}^\gamma B_{4,*}^-,        &    &
\end{aligned} 
\right.
\end{equation}
and $B^-_{4,*}$ is an absolute constant given in \cite[Theorem 3]{Lau_Wu}. 
On the other hand, they showed that there are $f_m^{\pm}\in \mathcal{H}_2^*(N)$ such that
\begin{equation}\label{ExtremValueCM}
L(s, {\rm sym}^mf_m^{\pm})\gtreqless \{1+o(1)\}(B^\pm_m\log_2N)^{\pm A_m^\pm}
\end{equation}
for all \textit{primes} $N\to\infty$.

Lau and Wu obtained the analogues of \eqref{CardementCM} and \eqref{ExtremValueCM} in the weight aspect
(see \cite[Theorem 2 and 3]{Lau_Wu}).
In order to prove these results, they showed that for $f, g\in \mathcal{H}_k^*(1)$ 
the archimedean local factor of the Rankin-Selberg $L$-function 
$L(s, {\rm sym}^mf\times {\rm sym}^mg)$ is
\begin{equation}\label{RS_GammaFactor}
L_\infty(s, {\rm sym}^mf\times {\rm sym}^mg)
= \Gamma_{\mathbb R}(s)^{\delta_{2\mid m}}
\Gamma_{\mathbb C}(s)^{[m/2]+\delta_{2\nmid m}}
\prod_{1\le \nu\le m} \Gamma_{\mathbb C}\big(s+\nu (k-1)\big)^{m-\nu+1}
\end{equation}
and established a density theorem on the zeros of $L(s, {\rm sym}^mf)$
in the weight aspect (see \cite[Proposition 2.1 and Theorem 1]{Lau_Wu}). 

In this paper, we shall study the distribution of $L(1, {\rm sym}^mf)$ in both aspects of level and weight
by refining the methods of \cite{Royer_Wu, Cogdell_Michel, Lau_Wu}.
The statements of our results are restricted to the symmetric first, second, third and fourth power 
because those are the ones currently known to be automorphic and cuspidal 
(for square free level $N$ and trivial nebentypus, where no CM forms or forms of weight 1 exist), 
but the method will apply for higher powers 
when automorphy and cuspidality become known.
Since we consider the level aspect and the weight aspect at the same time,
the situation will be more complicated.
In order to describe precisely the relation between the extreme values of $L(1, {\rm sym}^mf)$ 
in the level-weight aspect and arithmetic properties of $N$, 
for each positive constant $\Xi>0$ and even integer $k\ge 2$ we define the set of levels :
\begin{equation}\label{DefNari}
\mathbb{N}_k(\Xi) := \big\{N\in \mathbb{N} \,:\,  
\mu(N)^2=1 
\;\; \text{and} \;\;
P^-(N)\ge \Xi \log (kN)\log_2(kN)\big\},
\end{equation} 
where $\mu(n)$ is the M\"obius function.

Our first result is as follows.

\begin{theorem}\label{thm1}
Let $\Xi$ be a positive constant and $m=1, 2, 3, 4$. 
\par
{\rm (i)}
For $f\in \mathcal{H}^*_k(N)$, under Generalized Riemann Hypothesis (GRH) for $L(s, {\rm sym}^m f)$, we have
\begin{equation}\label{Thm1.EqA}
\{1+o(1)\}\big(2B^-_m\log_2 (kN)\big)^{-A^-_m}\leq L(1,{\rm sym}^m f)\leq \{1+o(1)\}\big(2B^+_m\log_2 (kN)\big)^{A^+_m}
\end{equation}
for $kN\to \infty$ with $2\mid k$ and $N\in \mathbb{N}_k(\Xi)$.
\par
{\rm (ii)}
There exist $f_m^{\pm}\in \mathcal{H}^*_k(N)$ such that
\begin{equation}\label{thm1.EqB}
L(s, {\rm sym}^mf_m^{\pm})\gtreqless \{1+o(1)\}\big(B^\pm_m\log_2(kN)\big)^{\pm A_m^\pm}
\end{equation}
for $kN\to \infty$ with $2\mid k$ and $N\in \mathbb{N}_k(\Xi)$.

Here $A^{\pm}_m$ and $B_m^{\pm}$ are defined as in \eqref{defAmBm}
$($see also \eqref{defAmpmBmpm} below$)$
and the implied constants depend on $\Xi$ only.
\end{theorem}

\begin{remark}
(i)
Taking $k=2$ in Theorem \ref{thm1}, 
we obtain a generalization of Codgell-Michel's \eqref{CardementCM} and \eqref{ExtremValueCM}, 
since $\mathbb{N}_2(\Xi)$ contains all primes for some suitable positive constant $\Xi$. 
\par
(ii)
Taking $N=1$ in Theorem \ref{thm1}, 
we can get Lau-Wu's corresponding results (see \cite[Theorem 2 and 3]{Lau_Wu}),
since $1\in \mathbb{N}_k(\Xi)$ for all even integers $k\ge 2$ and any positive constant $\Xi$.
\par
(iii)
As in \cite[Theorem 3(i)]{Lau_Wu}, we can prove that the bounds  
\begin{equation}\label{eq 30?}
(\log_2 (kN))^{-A_m^-}\ll L(1,{\rm sym}^m f)\ll (\log_2 (kN))^{A^+_m}
\end{equation}
holds unconditionally for almost all $f\in \mathcal{H}_k^*(N)$ and $1\le m\le 4$.
\end{remark}

\subsection{Montgomery-Vaughan's first conjecture}\

\vskip 1mm

Montgomery-Vaughan three conjectures describe very precisely the behavior of distribution functions of $L(1, \chi_d)$ around their extreme values \cite{Montgomery_Vaughan}. 
In this subsection, 
we consider the analogue of Montgomery-Vaughan's first conjecture for $L(1, {\rm sym}^mf)$.
For a fixed integer $m$, consider the distribution function  
\begin{equation}\label{Defprob}
\begin{aligned}
F_{k,N}^{\pm}(t,{\rm sym}^m)
:= \frac{1}{|\mathcal{H}^*_k(N)|}
\sum_{\substack{f\in \mathcal{H}^*_k(N)\\ L(1,{\rm sym}^m f)\gtreqless (B_m^\pm t)^{\pm A_m^\pm}}}1.
\end{aligned}
\end{equation}
In view of Theorem \ref{thm1},
the analogue of Montgomery-Vaughan's first conjecture for automorphic symmetric power $L$-functions can be stated as follows: 
\textit{For any fixed constant $\Xi>0$,
there are positive constants $c_2>c_1>c_0>0$ depending on $\Xi$ such that}
\begin{equation}\label{MV_Conjecture1}
{\rm e}^{-c_2(\log (kN)/\log_2 (kN)}
\leq F_{k,N}^{\pm}(\log_2 (kN),{\rm sym}^m)
\leq {\rm e}^{-c_1(\log (kN))/\log_2 (kN)}
\end{equation}
\textit{for $kN\ge c_0$ with $2\mid k$ and $N\in \mathbb{N}_k(\Xi)$.}

This problem was first studied by Lau and Wu \cite{Lau_Wu2}.
They proved the upper bound part of \eqref{MV_Conjecture1} when $N=1$ and $1\le m\le 4$:
\begin{equation}\label{MV_ConjectureLauWu}
F_{k, 1}^{\pm}(\log_2 k,{\rm sym}^m)
\leq {\rm e}^{-c_1(\log k)/\log_2k}
\end{equation}
for all even integers $k\ge c_0$.
It is quite remarkable that, 
despite the difficulties in handling modular forms 
as efficiently as Dirichlet characters, 
this result is almost as good as those of Granville and Soundararajan \cite{Granville_Soundararajan} 
in this other case (moreover, they use a different method at crucial points, 
where tools such as the Graham-Ringrose bounds
for short character sums with highly composite moduli are unavailable). 
The main tool is their large sieve inequality (see \cite[Theorem 1]{Lau_Wu2} or Lemma \ref{LargeSieve2} below), 
which also is quite likely to have other uses in this field. 
About the lower bound part of Montgomery-Vaughan's conjecture \eqref{MV_Conjecture1},
Liu, Royer and Wu \cite{Liu_Royer_Wu} obtained a slightly weaker result for $m=1$ and $N=1$ :
There are three absolute constants $c_3$, $c_4$ and $c_5$ such that 
\begin{equation}\label{MV_ConjectureLiuRoyerWu}
F_{k,1}^{\pm}\big(\log_2k-\tfrac{5}{2}\log_3k-\log_4k-c_3, {\rm sym}^1\big)
\geq {\rm exp}\bigg(-c_4\frac{\log k}{(\log_2k)^{7/2} \log_3k}\bigg)
\end{equation}
for $k\ge c_5$.

We shall generalize and/or improve \eqref{MV_ConjectureLauWu} and \eqref{MV_ConjectureLiuRoyerWu} as follows.

\begin{theorem}\label{thm2}
Let $\Xi$ be a positive constant and $m = 1, 2, 3, 4$. 
\par
{\rm (i)}
For any $\varepsilon>0$, there are positive constants $c_6$ and $c_7$ depending on $\varepsilon$ and $\Xi$ such that
\begin{equation*}
F^{\pm}_{k,N}\big(\log_2 (kN)+\phi, {\rm sym}^m\big)
\leq {\rm exp}\bigg(-c_6(|\phi|+1)\frac{\log (kN)}{\log_2 (kN)}\bigg)
\end{equation*}
for $kN\geq c_7$ with $2\mid k$ and $N\in \mathbb{N}_k(\Xi)$ 
and $\log \varepsilon \leq \phi\leq 9\log_2 (kN)$.
\par
{\rm (ii)}
There are positive constants $c_8$, $c_9$ and $c_{10}$ depending on $\Xi$ such that 
\begin{equation*}
F_{k,N}^{\pm}\big(\log_2(kN)-\log_3(kN)-\log_4(kN)-c_8, {\rm sym}^m\big)
\geq {\rm exp}\bigg(-\frac{c_9\log(kN)}{\log_2^2(kN) \log_3(kN)}\bigg)
\end{equation*}
for $kN\ge c_{10}$ with $2\mid k$ and $N\in \mathbb{N}_k(\Xi)$.
\end{theorem}

\begin{remark}
(i)
Taking $\phi=0$ in Theorem \ref{thm2}(i), we get the upper bound part of Montgomery-Vaughan's first conjecture \eqref{MV_Conjecture1} in the level-weight aspect.
\par
(ii)
Theorem \ref{thm2}(ii) can be regarded as a weak version of the lower bound part of Montgomery-Vaughan's first conjecture \eqref{MV_Conjecture1}.
\par
(iii)
Since $1\in \mathbb{N}_k(\Xi)$ for all even integers $k\ge 2$ and all positive constants $\Xi$,
it is easy to see that
Theorem \ref{thm2}(i) and (ii) generalise and improve 
\eqref{MV_ConjectureLauWu} of Lau-Wu and/or \eqref{MV_ConjectureLiuRoyerWu} of Liu-Royer-Wu, respectively.
\end{remark}

\subsection{Weighted distribution functions}\

\vskip 1mm

Motivated by the works of Granville-Soundararajan \cite{Granville_Soundararajan} 
and of Cogdell-Michel \cite{Cogdell_Michel}
and in view of the Petersson trace formula, 
Liu, Royer and Wu \cite{Liu_Royer_Wu} introduced the weighted distribution functions :
\begin{equation}\label{defharmprob1}
\mathscr{F}_{k,N}^{\pm}(t, {\rm sym}^m)
:= \frac{1}{\displaystyle\sum_{f\in \mathcal{H}_k^{*}(N)}\omega_f}
\sum_{\substack{f\in \mathcal{H}^*_k(N)\\ L(1,{\rm sym}^m f)\gtreqless (B_m^\pm t)^{\pm A_m^\pm}}} \omega_f,
\end{equation}
where $\omega_f$ is defined as in \eqref{harmweight}.
By using the saddle-point method, they evaluated \eqref{defharmprob1} for $N=m=1$: 
There are three positive constants $\mathscr{A}_1^{\pm}$ and $C$ such that we have, for $k\to\infty$,
\begin{equation}\label{DF_LiuRoyerWu}
\mathscr{F}_{k, 1}^{\pm}(t, {\rm sym}^1)
= \{1+o(1)\} \exp\bigg(-\frac{{\rm e}^{t-\mathscr{A}_1^{\pm}}}{t}\bigg\{1+O\bigg(\frac{1}{t}\bigg)\bigg\}\bigg),
\end{equation}
uniformly for $t\leqslant \log_2k-\tfrac{5}{2}\log_3k-\log_4 k-C$,
where the implied constant is absolute.
As they noted, their method should work in the symmetric power case but with additional technical issues.
In \cite{Lamzouri}, Lamzouri studied a large class of random Euler products and 
gave a quite general result \cite[Theorem 1]{Lamzouri}. 
As a corollary, he obtained the evaluation of \eqref{defharmprob1} with sign $+$ and $k=2$ in the \textit{prime} level aspect:
\begin{equation}\label{DF_Lamzouri}
\mathscr{F}_{2, N}^{+}(t, {\rm sym}^m)
= \{1+o(1)\} \exp\bigg(-\frac{{\rm e}^{t-\mathscr{A}_m^{+}}}{t}\bigg\{1+O\bigg(\frac{1}{\sqrt{t}}\bigg)\bigg\}\bigg)
\end{equation}
uniformly for all \textit{prime} numbers $N$ and $t\leqslant \log_2N-\log_3N-2\log_4N$.
We note that the domain of validity of $t$ is slightly lager than that of \eqref{DF_LiuRoyerWu}
but the error term is slightly weaker than that of \eqref{DF_LiuRoyerWu}.

By refining Lamzouri's method \cite{Lamzouri}, we can prove the following result.

\begin{theorem}\label{thm3} 
Let $\Xi$ be a positive constant and $m = 1, 2, 3, 4$. 
Then there is a positive constant $c_{11}$ depending on $\Xi$ such that we have
\begin{equation*}
\mathscr{F}_{k,N}^{\pm}(t, {\rm sym}^m)
= \{1+o(1)\}
\exp\bigg(-\frac{{\rm e}^{t-\mathscr{A}_{m}^{\pm}}}{t}\bigg\{1+O\bigg(\frac{1}{t}\bigg)\bigg\}\bigg)
\end{equation*}
uniformly for $kN\to\infty$ with $2\mid k$ and $N\in \mathbb{N}_k(\Xi)$ and 
$$
t\leqslant \log_2(kN)-\log_3(kN)-\log_4(kN)-c_{11},
$$
where $\mathscr{A}_m^{\pm}$ are constants depending only on $m$ defined as in Lemma \ref{lem7.2} below.
Here the implied constants depend on $\Xi$ only. 
\end{theorem}

\begin{remark}
(i)
Clearly Theorem \ref{thm3} generalizes and improves \eqref{DF_LiuRoyerWu} of Liu-Royer-Wu
and \eqref{DF_Lamzouri} of Lamzouri.
\par
(ii)
Theorem \ref{thm3} also completes \eqref{DF_Lamzouri} of Lamzouri by proving similar result in the case of sign $-$.
\end{remark}

According to \eqref{HoffLock},
it is not difficult to see that
\begin{equation}\label{1.01}
\mathscr{F}_{k,N}^{\pm}(t, {\rm sym}^m)/\log (kN)
\ll F^{\pm}_{k,N}(t,{\rm sym}^m)
\ll \mathscr{F}_{k,N}^{\pm}(t, {\rm sym}^m)\log (kN)
\end{equation}
for all even integers $k\ge 2$, square free integers $N\ge 1$ and real numbers $t>0$,
where the implied constants are absolute.
From Theorem \ref{thm3}, we immediately deduce the following corollary.


\begin{corollary}\label{Cor} 
Let $\Xi$ be a positive constant and $m = 1, 2, 3, 4$. 
There exist four positive constants $c_{11}, c_{12},c_{13}, c_{14}$ depending on $\Xi$ only such that
$${\rm e}^{-c_{12}\log (kN)/(\log_2^2(kN)\log_3 (kN))}
\ll F^{\pm}_{k, N}(T_{k, N}, {\rm sym}^m)
\ll {\rm e}^{-c_{13}\log (kN)/(\log_2^2(kN)\log_3 (kN))}.$$
for $kN\ge c_{14}$ with $2\mid k$ and $N\in \mathbb{N}_k(\Xi)$, 
where $T_{k, N} := \log_2(kN)-\log_3(kN)-\log_4(kN)-c_{11}$.
\end{corollary}

\subsection{Density theorem on symmetric power $L$-functions in the level-weight aspect}\

\vskip 1mm

In the methods of \cite{Royer_Wu, Cogdell_Michel, Lau_Wu}, theorem of density plays a key role.
A rather general density theorem on automorphic $L$-functions in the level aspect was established by Kowalski and Michel \cite[Theorem 2]{Kowalski_Michel} and used in \cite{Royer_Wu, Cogdell_Michel}.
A similar density result in the weight aspect was obtained by Lau and Wu \cite[Theorem 1]{Lau_Wu}.
In order to prove our Theorem \ref{thm1}, 
it is necessary to establish a density theorem on symmetric power $L$-functions in the level-weight aspect.
Denote $N(\alpha,T,{\rm sym}^m f)$ the number of zeros $\rho=\beta+{\rm i}\gamma$ of $L(s,{\rm sym}^mf)$ with $\beta\geqslant \alpha$ and $0\leqslant \gamma\leqslant T$.

Our density theorem is as follows.

\begin{theorem}\label{thm4}
Let $\alpha>\tfrac{1}{2}$, $\varepsilon>0$, $1\le m\le 4$, $r>0$, $E_{m,r}=(m+1)(m+r)+8$ and $E_{m, r}'=(2m+r)(m+1)+m+12$.
Then we have
\begin{equation*}
\sum_{f\in \mathcal{H}^*_k(N)}N(\alpha,T,{\rm sym}^m f)
\ll_{\alpha, \varepsilon, r} T^{1+1/r}k^{E_{m,r}(1-\alpha)/(3-2\alpha)+\varepsilon}N^{E_{m, r}'(1-\alpha)/(3-2\alpha)+\varepsilon},
\end{equation*}
uniformly for $2\mid k$, square free $N$ and $T\geq 2$,
where the implied constant depends only on $\alpha$, $\varepsilon$ and $r$. 
\end{theorem}

The density theorem shows that on average over the family $\mathcal{H}^*_k(N)$ there are very few forms with zeros in the critical strip with real part near the line $\re s=1$.
This theorem is useful only when $\alpha$ is very close to $1$ and the $T$-aspect is essentially irrelevant.
For $\eta \in (0, \tfrac{1}{2})$, define
\begin{equation}\label{e 15}
\begin{aligned}
\mathcal{H}^+_k(N;\eta,m)
& := \{f\in \mathcal{H}^*_k(N): L(s,{\rm sym}^m f)\neq 0, s \in \mathcal{S}\},
\\
\mathcal{H}^-_k(N; \eta, m)
& := \mathcal{H}^*_k(N)\setminus \mathcal{H}^+_k(N;\eta,m),
\end{aligned}
\end{equation}
where
$$
\mathcal{S} := \{s:1-\eta\leq \sigma<1, |\tau|\leq 100(kN)^{\eta}\}\cup {\{s:\sigma\geq 1\}}.
$$
By using Theorem \ref{thm4} with $r=1$, we have
\begin{equation}\label{2.12-}
\begin{aligned}
\mathcal{H}^-_k(N; \eta, m)
& \leq \sum_{f\in \mathcal{H}^-_k(N; \eta, m)} N(1-\eta, 100(kN)^\eta, {\rm sym}^m f)
\\
& \leq \sum_{f\in \mathcal{H}^*_k(N)}N(1-\eta, 100(kN)^\eta, {\rm sym}^m f)\ll_{\eta} (kN)^{65\eta}.
\end{aligned}
\end{equation}
For $\eta<\tfrac{1}{65}$, we have
\begin{equation}\label{2.12+}
|\mathcal{H}^+_k(N; \eta,m)|\sim |\mathcal{H}^*_k(N)|.
\end{equation}
As $\mathcal{H}^+_k(N; \eta, m)$ has almost the same size as $\mathcal{H}^*_k(N)$, 
we replace $\mathcal{H}^*_k(N)$ by $\mathcal{H}^+_k(N; \eta, m)$ in the applications and 
the density theorem can partially play the role of Generalized Riemann Hypothesis.

\vskip 8mm

\section{Some lemmas}

In this section, we shall establish some unconditional and conditional bounds of $L(s, {\rm sym}^mf)$
in the critical strip, which will be useful later.

\subsection{Automorphic $L$-functions and convexity bounds}\

\vskip 1mm

The $m$-th symmetric power $L$-function attached to $f\in \mathcal{H}_k^*(N)$ defined as in \eqref{defLssymmf}
has the Dirichlet series for $\sigma>1$,
$$L(s,{\rm sym}^m f)=\sum_{n=1}^\infty \lambda_{{\rm sym}^m f}(n)n^{-s},$$
where $\lambda_{{\rm sym}^m f}(n)$ is multiplicative and admits
\begin{equation}\label{eq 2}
|\lambda_{{\rm sym}^m f} (n)|\leq d_{m+1}(n)
\qquad
(n\ge 1).
\end{equation}
Here $d_2(n)=d(n)$ and $d_{m+1}(n):=\sum_{\ell\mid n} d_m(\ell)$.

The symmetric $L$-function has the degree $d=m+1$, the conductor ${\rm Cond}({\rm sym}^m f)=N^m$ and extends to an entire function on $\mathbb{C}$ by the functional equation given in the next section without any poles.

For $m\in \mathbb{N}$ and $f,g\in \mathcal{H}^*_k(N)$, the Rankin-Selberg $L$-function of ${\rm sym}^m f$ and ${\rm sym}^m g$ is defined by
\begin{equation}\label{defLssymmfsymmg}
L(s,{\rm sym}^m f\times {\rm sym}^m g)
:= \prod_p\prod_{0\leq i,j\leq m}(1-\alpha_f(p)^{m-i}\beta_f(p)^i\alpha_g(p)^{m-j}\beta_g(p)^jp^{-s})^{-1},
\end{equation}
with Dirichlet series expansion
$$
\sum_{n=1}^{\infty} \lambda_{{\rm sym}^m f\times{\rm sym}^m g}(n)n^{-s}.
$$
It extends to a meromorphic function on $\mathbb{C}$ which has no pole except possibly at $s=1$ if and only if when $f=\bar{g}$. What's more, we have
$$
\lambda_{{\rm sym}^m f\times{\rm sym}^m g}(p)=\lambda_{{\rm sym}^m f}(p)\lambda_{{\rm sym}^m g}(p),
$$
for unramified $p\nmid N$.
The conductor of Rankin-Selberg $L$-function of ${\rm sym}^m f$ and ${\rm sym}^m g$ denoted by ${\rm Cond}({\rm sym}^m f\times{\rm sym}^m g)$ satisfies (see \cite{{Bushnell_Henniart}})
\begin{equation*}
{\rm Cond}({\rm sym}^m f\times{\rm sym}^m g)\leq ({\rm Cond}({\rm sym}^m f){\rm Cond}({\rm sym}^m g))^{m+1}=N^{2m(m+1)}.
\end{equation*}
Let $L_\infty(s,{\rm sym}^m f \times {\rm sym}^m g)$ be the archimedean local factor given as in \eqref{RS_GammaFactor}.
The complete symmetric power Rankin-Selberg $L$-function
\begin{align*}
& \Lambda(s, {\rm sym}^m f \times {\rm sym}^m g)
\\
& := {\rm Cond}({\rm sym}^m f\times{\rm sym}^m g)^{s/2}
L_\infty(s,{\rm sym}^m f \times {\rm sym}^m g)L(s,{\rm sym}^m f \times {\rm sym}^m g)
\end{align*}
satisfies the functional equation
\begin{equation*}
\Lambda(s,{\rm sym}^m f \times {\rm sym}^m g)=\varepsilon_{{\rm sym}^m f \times {\rm sym}^m g}\Lambda(1-s,{\rm sym}^m f \times {\rm sym}^m g)
\qquad
(s\in {\mathbb C})
\end{equation*}
with $\varepsilon_{{\rm sym}^m f \times {\rm sym}^m g}=\pm 1$.

We denote the special Rankin-Selberg $L$-function
\begin{equation}\label{defSRMfunction}
\mathcal{L}(s,{\rm sym}^m f \times {\rm sym}^m g)
:= \sum_{n\geq 1}\lambda_{{\rm sym}^m f}(n)\lambda_{{\rm sym}^m g}(n)n^{-s}.
\end{equation}

We have the convexity bounds for these automorphic $L$-functions.

\begin{lemma}\label{lem2.1}
Let $1\le m\le 4$, $2\mid k$, $N$ be square free and $f,g\in \mathcal{H}^*_k(N)$. For $0\leq \sigma \leq 1$ and any $\varepsilon>0$, we have 
\begin{equation}\label{ConvexityBound}
L(s, {\rm sym}^m f)
\ll \left\{
\begin{aligned}
& N^{m(1-\sigma)/2+\varepsilon}(k+|\tau|)^{([m/2]+1)(1-\sigma)+\varepsilon}                                 && \mbox{if $2\nmid m$}
\\\noalign{\vskip 1mm}
& N^{m(1-\sigma)/2+\varepsilon}(1+|\tau|)^{(1-\sigma)/2}(k+|\tau|)^{[m/2](1-\sigma)+\varepsilon}  && \mbox{if $2\mid m$}
\end{aligned}
\right.
\end{equation}
and
\begin{align}
L(s, {\rm sym}^m f \times {\rm sym}^m g)   
& \ll N^{m(m+1)(1-\sigma)+\varepsilon} (1+|\tau|)^{(m+1)(1-\sigma)/2} (k+|\tau|)^{m(m+1)(1-\sigma)/2+\varepsilon}
\label{RS_ConvexityBound}
\\\noalign{\vskip 1mm}
\mathcal{L}(s, {\rm sym}^m f \times {\rm sym}^m g)
& \ll N^{m(m+1)(1-\sigma)+\varepsilon}(1+|\tau|)^{(m+1)(1-\sigma)/2}(k+|\tau|)^{m(m+1)(1-\sigma)/2+\varepsilon}
\label{Special_RS_ConvexityBound}
\end{align}
where the implied constants depend on $\varepsilon$ only.
\end{lemma}

By \eqref{defLssymmf}, we write the Dirichlet series of logarithmic derivative as
\begin{equation}\label{logder}
-\frac{L'}{L}(s, {\rm sym}^m f)=\sum_{n=1}^{\infty} \frac{\Lambda_{{\rm sym}^m f}(n)}{n^s}
\end{equation}
for $\sigma>1$, where
\begin{equation}\label{e 8}
\Lambda_{{\rm sym}^m f}(n)=
\left\{
\begin{aligned}
&  \alpha_f(p)^{m\nu}\log p           && \mbox{if $n=p^{\nu}$ and $p\mid N$},\\
& [\alpha_f(p)^{m\nu}+\alpha_f(p)^{(m-2)\nu}+\cdots+\alpha_f(p)^{-m\nu}]\log p &&\mbox{if $n=p^{\nu}$ and $p\nmid N$},\\
&0  &&\mbox{otherwise}.
\end{aligned}
\right.
\end{equation}
It is apparent that $|\Lambda_{{\rm sym}^m f}(n)|\leq (m+1)\log n$ for $n>1$.

\subsection{Bounds for symmetric power $L$-functions}\

\vskip 1mm

The following proposition about bounds for symmetric power $L$-functions will be needed later.

\begin{lemma}\label{lem2.2}
For $1\le m\le 4$, $2\mid k$, square free $N$ and $f\in \mathcal{H}_k^*(N)$, we have 
\begin{equation*}
L(s, {\rm sym}^m f)\ll \log^{m+1}(N(k+|s|+2))
\end{equation*}
uniformly for $\re s\geq 1-1/\log(N(k+|s|+2))$.
\end{lemma}

\begin{proof}
It suffices to consider $\tfrac{3}{2}\geq \re s\geq 1-1/\log(N(k+|s|+2))$. 
According to the Perron formula, and by standard contour shifts and \eqref{ConvexityBound} of Lemma \ref{lem2.1}, 
we have for any $\varepsilon>0$,
\begin{align*}
\sum_{n\geq 1}\frac{\lambda_{{\rm sym}^m f}(n)}{n^s}{\rm e}^{-n/Y}
& = \frac{1}{2\pi {\rm i}}\int_{(2)}L(u+s,{\rm sym}^m f)Y^u \Gamma(u)\,\dd{u}
\\\noalign{\vskip -1,5mm}
& = L(s,{\rm sym}^m f)+ \frac{1}{2\pi {\rm i}} \int_{(\frac{1}{2}-\re s)}L(u+s,{\rm sym}^m f)\Gamma(u)Y^u\,\dd{u}
\\\noalign{\vskip 0mm}
& = L(s,{\rm sym}^m f)+O\big(N^{m/4+\varepsilon}(|s|+k)^{(m+1)/4+\varepsilon}Y^{1/2-\re s}\big).
\end{align*}
Taking $Y=N^{m/2+1}(|s|+k)^{(m+1)/2+1}$ and using \eqref{eq 2} 
we get the result by the bound of zeta function near the line $\re s=1$.
\end{proof}

For $f\in \mathcal{H}^+_k(N;\eta,m)$, where $\eta\in (0,\tfrac{1}{2})$, 
we get the logarithm $\log L(s,{\rm sym}^m f)$ from the integral of logarithmic derivative \eqref{logder} since it is holomorphic and has no zero in the region $\mathcal{S}$ defined in \eqref{e 15}. That is 
\begin{equation}\label{e 20}
\log L(s,{\rm sym}^m f)=\sum_{n=1}^{\infty} \frac{\Lambda_{{\rm sym}^m f}(n)}{n^s\log n} \qquad (\sigma>1).
\end{equation}
Immediately we get the simple bound for $\log L(s,{\rm sym}^m f)$,
\begin{equation}\label{eq 17}
|\log L(s,{\rm sym}^m f)|\leq(m+1)\zeta(\sigma)\ll_m(\sigma-1)^{-1} \qquad (\sigma>1).
\end{equation}
Let us write $\sigma_0=1-\eta$. With the convexity bound and the Borel-Carathedory theorem, we also have for $\sigma>\sigma_0$ and $|\tau|\leq 100(kN)^{\eta}$,
\begin{equation}\label{eq 12}
\log L(s,{\rm sym}^m f)\ll \tfrac{\log (kN)}{\sigma-\sigma_0}.
\end{equation}

The following lemma gives an upper bound of $\log L(s,{\rm sym}^m f)$ under GRH.

\begin{lemma}\label{lem2.3}
Let $1\le m\le 4$, $2\mid k$, $N$ be square free and $f\in \mathcal{H}^*_k(N)$. 
Under GRH for $L(s,{\rm sym}^m f)$, we have for any $\varepsilon>0$ and any $\alpha>\tfrac{1}{2}$,
\begin{equation*}
\log L(s,{\rm sym}^m f)\ll_{\varepsilon,\alpha}[\log(N(k+|s|+3))]^{2(1-\sigma)+\varepsilon}
\end{equation*}
uniformly for $\alpha\leq\sigma\leq 1$ and $\tau\in \mathbb{R}$.
\end{lemma}

\begin{proof}
We denote $F(s):=\log L(s,{\rm sym}^m f)$.
Under GRH for $L(s,{\rm sym}^m f)$, $F(s)$ is holomorphic for $\re s>\tfrac{1}{2}$.
With the convexity bound of \eqref{ConvexityBound}, we have 
\begin{equation*}
\re \log L(s,{\rm sym}^m f)\leq C\log(N(k+|\tau|+3)) 
\qquad 
(\sigma>\tfrac{1}{2}).
\end{equation*}
Applying the Borel-Caratheodory theorem, 
we choose $s'=2+\text{i}\tau,\, R'=\tfrac{3}{2}-\tfrac{1}{2}\delta$ and $r'=\tfrac{3}{2}-\delta$, where $0<\delta < 1$ will be chosen later. Then we have
\begin{equation*}
\begin{aligned}
\max_{|s-s'|=r'} |F(s)|
& \leq \frac{2r'}{R'-r'}\max_{|s-s'|=R'}\re F(s)+\frac{R'+r'}{R'-r'}|F(s')|
\\
& \leq (6/\delta-4)C\log N(k+|\tau|+3)+(6/\delta-3)C
\\\noalign{\vskip 1mm}
& \leq C\delta^{-1} \log(N(k+|\tau|+3)).
\end{aligned}
\end{equation*}
So for $\delta+\tfrac{1}{2}\leq \re s \leq \tfrac{7}{2}-\delta$, we have
\begin{equation}\label{eq 11}
|F(s)|\leq C\delta^{-1} \log(N(k+|\tau|+3)).
\end{equation}
Denote $M(r):=\max\limits_{|s-s_0|=r}|F(s)|$.
Applying the Hadamard three circle theorem with the center $s_0=\sigma_1+{\rm i}\tau \; (1<\sigma_1\leq N(k+|\tau|+3))$ 
and $r_1=\sigma_1-1-\delta$, $r_2=\sigma_1-\sigma$, $r_3=\sigma_1-\tfrac{1}{2}-\delta$, we have
$$
M(r_2)\leq M(r_1)^{1-a}M(r_3)^a
\quad\text{with}\quad
a=\tfrac{\log(r_2/r_1)}{\log(r_3/r_1)}=2(1-\sigma)+O(\delta+1/\sigma_1).
$$
Thanks to \eqref{eq 11}, we have
$M(r_3)\leq C\delta^{-1} \log(N(k+|\tau|+3))$ and $M(r_1)\leq C\delta^{-1}$.
Therefore we obtain
\begin{equation*}
|\log L(s,{\rm sym}^m f)|\leq \left(C\delta^{-1}\right)^{1-a} \left(C\delta^{-1} \log(N(k+|\tau|+3))\right)^a.
\end{equation*}
At last we choose $\sigma_1=\tfrac{1}{\delta}=\log_2 N(k+|\tau|+3)$,
then we get our result.
\end{proof}

We get a better bound than \eqref{eq 12} without GRH when $f\in \mathcal{H}^+_k(N; \eta, m)$.

\begin{lemma}\label{lem2.4}
Let $\eta\in (0, \tfrac{1}{2})$ fixed, $\sigma_0=1-\eta$, $1\le m\le 4$, $2\mid k$ and $N$ be square free. We have for $f\in \mathcal{H}^+_k(N;\eta,m)$,
\begin{equation}\label{eq 13}
\log L(s,{\rm sym}^m f)
= \sum_{n=1}^{\infty} \frac{\Lambda_{{\rm sym}^m f}(n)}{n^s\log n}{\rm e}^{-n/T}+R
\end{equation}
uniformly for $3\leq T\leq (kN)^{\eta}$, $\sigma_0<\sigma\leq 3/2$ and $|\tau|\leq T$, where
\begin{equation}\label{eq 14}
R\ll_{\eta}T^{-(\sigma-\sigma_0)/2}(\log (kN))/(\sigma-\sigma_0)^2.
\end{equation}
Furthermore for any $0<\varepsilon<\tfrac{1}{4}$ and $\tfrac{1}{2}<\alpha<1$, 
under GRH for $L(s,{\rm sym}^m f)$ where $f\in \mathcal{H}^*_k(N)$, 
the formula \eqref{eq 13} holds uniformly for $\alpha\leq \sigma\leq \tfrac{3}{2}$ and $T\geq 1$, with
\begin{equation*}
R\ll_{\varepsilon,\alpha}T^{-(\sigma-\alpha)}(\log (kN))^{2(1-\alpha)+\varepsilon}.
\end{equation*}
\end{lemma}

\begin{proof}
We have
\begin{equation*}
\sum_{n=2}^{\infty}\frac{\Lambda_{{\rm sym}^m f}(n)}{n^s\log n}{\rm e}^{-n/T}=\frac{1}{2\pi {\rm i}}\int_{2-{\rm i}\infty}^{2+{\rm i}\infty}\Gamma(z-s)\log L(z,{\rm sym}^m f)T^{z-s}\,\dd{z}.
\end{equation*}
Shifting the line of integral to the path $\mathcal{C}$ consisting of the straight lines joining
$$\kappa-{\rm i}\infty,\quad \kappa-2{\rm i}T, \quad \sigma_1-2{\rm i}T,  \quad \sigma_1+2{\rm i}T, \quad \kappa+2{\rm i}T,\quad \kappa+{\rm i}\infty,$$
where $\kappa=1+1/\log T$ and $\sigma_1=(\sigma+\sigma_0)/2$, we have
\begin{equation*}
\sum_{n=2}^{\infty}\frac{\Lambda_{{\rm sym}^m f}(n)}{n^s\log n}{\rm e}^{-n/T}=\log L(s,{\rm sym}^m f)+\frac{1}{2\pi{\rm i}}\int_{\mathcal{C}}\Gamma(z-s)\log L(z,{\rm sym}^m f)T^{z-s}\,\dd{z}.
\end{equation*}
By \eqref{eq 17} and \eqref{eq 12}, the last integral is 
\begin{equation*}
\begin{aligned}
& \ll \frac{T^{\sigma_1-\sigma}\log (kN)}{\sigma-\sigma_0}\int_{|y|\leq 3T}|\Gamma(\sigma_1-\sigma+{\rm i}y)|\,\dd{y}
\\
& \quad
+\frac{\log (kN)}{\sigma-\sigma_0}\int_{\sigma_1}^{\kappa}T^{x-\sigma}|\Gamma(x-\sigma+{\rm i}(T-\tau))|\,\dd{x}
+T^{1-\sigma+\varepsilon}\int_{|y|\geq T}|\Gamma(\kappa-\sigma+{\rm i}y)|\,\dd{y}.
\end{aligned}
\end{equation*}
Then we can get \eqref{eq 13} and \eqref{eq 14} with the Stirling formula.
Under GRH, we use the same method and shift the line of integration to $\re z=\alpha'=\alpha-\varepsilon'>\tfrac{1}{2}$ 
where $\varepsilon'=\tfrac{1}{2}\min(\varepsilon,\alpha-\tfrac{1}{2})$. Then the last integral will be
\begin{equation*}
\begin{aligned}
R&=\frac{1}{2\pi{\rm i}}\int_{\alpha'-{\rm i}\infty}^{\alpha'+{\rm i}\infty}\Gamma(z-s)\log L(z,{\rm sym}^m f)T^{z-s}\,\dd{z}\\
 &\ll_{\varepsilon,\alpha}T^{\alpha'-\sigma}\int_{-\infty}^{+\infty}|\Gamma(\alpha'-\sigma+{\rm i}y)|(\log N(k+|y|+3))^{2(1-\alpha')+\varepsilon}\,\dd{y}\\
 \noalign{\vskip 1mm}
 &\ll_{\varepsilon,\alpha} T^{-(\sigma-\alpha)}(\log (kN))^{2(1-\alpha)+2\varepsilon},
\end{aligned}
\end{equation*}
according to Lemma \ref{lem2.3}.
Then we complete the proof of the lemma.
\end{proof}

\begin{lemma}\label{lem2.5}
Let $\eta\in(0, \tfrac{1}{2})$ fixed, $1\le m\le 4$, $2\mid k$ and $N$ be square free. For any $f\in \mathcal{H}^+_{k}(N; \eta,m)$, we have
\begin{equation*}
\log L(s,{\rm sym}^m f)\ll_{\eta} \frac{(\log (kN))^{4\alpha/\eta}-1}{\alpha\log_2 (kN)}+\log_3(10kN),
\end{equation*}
uniformly for $\sigma\geq 1-\alpha>1-\tfrac{1}{2}\eta$ and $|\tau|\leq (\log (kN))^{4/\eta}$.
\end{lemma}

\begin{proof}
We take $T=(\log (kN))^{4/\eta}$ in Lemma \ref{lem2.4}, then the error term will be $O(1)$. For the summation, it is
$$\ll\sum_p p^{-\sigma}{\rm e}^{-p/T}+O(1).$$
Divide the summation into two parts,
\begin{equation*}
\sum_p p^{-\sigma}{\rm e}^{-p/T}\leq \sum_{p\leq T}p^{-\sigma}+\sum_{p>T} p^{-\sigma}{\rm e}^{-p/T}.
\end{equation*}
For the first sum, it is
\begin{equation*}
\begin{aligned}
\ll \sum_{p\leq T} \frac{1}{p^{1-\alpha}}
\ll_{\eta} \frac{(\log (kN))^{4\alpha/\eta}-1}{\alpha\log_2 (kN)}+\log_3(10kN).
\end{aligned}
\end{equation*}
Here we have used the fact that for $1/2\leq \sigma \leq 1$ and $y\ge 3$
$$\sum_{p\leq y}\frac{1}{p^{\sigma}} 
\ll \frac{y^{1-\sigma}-1}{(1-\sigma)\log y}+\log_2 y.$$
For the second sum, we have
\begin{equation*}
\begin{aligned}
\sum_{p>T}p^{-\sigma}{\rm e}^{-p/T}
& \ll \int_T^{\infty} {\rm e}^{-t/T}\,\dd\Big(\sum_{p\leq t} p^{-\sigma}\Big)
\\
& \ll \frac{T^{1-\sigma}-1}{(1-\sigma)\log T}+\log_2 T
+ \frac{1}{T}\int_T^{\infty} \bigg({\rm e}^{-t/T}\frac{t^{1-\sigma}-1}{(1-\sigma)\log t}+{\rm e}^{-t/T}\log_2t\bigg) \dd{t}
\\
& \ll_{\eta} \frac{(\log (kN))^{4\alpha/\eta}-1}{\alpha\log_2 (kN)}+\log_3(10kN).
\end{aligned}
\end{equation*}
Then we get our result.
\end{proof}

With the bound above, we can write the logarithm of symmetric $L$-functions as the following Dirichlet series.

\begin{lemma}\label{lem2.6}
Let $\eta \in (0, \tfrac{1}{65})$, $1\le m\le 4$, $2\mid k$ and $N$ be a square free number. 
Let $x=\exp \{\sqrt{\log (kN)/7(m+4)}\}$. Then we have
\begin{equation}\label{2.8}
\begin{aligned}
\log L(1,{\rm sym}^m f)
& = \sum_{\substack{p\leq x\\p\nmid N}}\sum_{0\leq j\leq m} \log \left(1-\alpha_f(p)^{m-2j}p^{-1}\right)^{-1}
\\
& \quad
+ \sum_{\substack{p\leq x\\p\mid N}}\log \left(1-\alpha_f(p)^mp^{-1}\right)^{-1}
+ O\big(\log^{-1/2}(kN)\big),
\end{aligned}
\end{equation}
for $f \in \mathcal{H}_{k}^+(N;\eta,m)$. The implied constant depends on $\eta$ and $m$.
\end{lemma}

\begin{proof}
Let $T=(\log (kN))^{4/\eta}$. In view of \eqref{e 20}, we have according to Perron formula
\begin{equation*}
\sum_{2\leq n \leq x} \frac{\Lambda_{{\rm sym}^m f}(n)}{n\log n}
=\frac{1}{2\pi {\rm i}} \int_{1/\log x -{\rm i}T}^{1/\log x +{\rm i}T}\log L(s+1, {\rm sym}^m f)\frac{x^s}{s}\,\dd{s}
+ O\bigg(\frac{\log (Tx)}{T}+\frac{1}{x}\bigg).
\end{equation*}
Move the integration to $\sigma=-\tfrac{1}{4}\eta$, 
and estimate $\log L(s, {\rm sym}^m f)$ by Lemma \ref{lem2.5} (with $\alpha=\tfrac{1}{4}\eta$), then we obtain
\begin{equation}\label{2.7}
\begin{aligned}
\sum_{2\leq n \leq x} \frac{\Lambda_{{\rm sym}^m f}(n)}{n\log n}
& = \log L(1,{\rm sym}^m f)+O\left(\frac{\log (kNTx)}{T}+\frac{\log (kN)\log T}{x^{\eta/4}}\right)
\\
& = \log L(1,{\rm sym}^m f)+O\big((\log (kN))^{-4/\eta+1}\big).
\end{aligned}
\end{equation}
On the other hand, \eqref{e 8} allows us to deduce 
\begin{equation*}
\begin{aligned}
\sum_{2\leq n \leq x} \frac{\Lambda_{{\rm sym}^m f}(n)}{n\log n}
& = \sum_{p\leq x} \sum_{\nu\leq \log x/log p} \frac{\Lambda_{{\rm sym}^m f}(p^\nu)}{p^\nu\log p^\nu}
\\
& = \sum_{\substack{p\leq x\\p\mid N}} \sum_{\nu\leq \log x/log p} \frac{\alpha_f(p)^{m\nu}}{\nu p^\nu}
+ \sum_{\substack{p\leq x\\p\nmid N}}\sum_{\nu \leq \log x/log p} \sum_{0\leq j\leq m} \frac{\alpha_f(p)^{\nu(m-2j)}}{\nu p^\nu}
\\
& = \sum_{\substack{p\leq x\\p\mid N}} \bigg\{\log\left(1-\frac{\alpha_f(p)^m}{p}\right)^{-1}+O\left(\frac{\log p}{x^{3/2}\log x}\right)\bigg\}
\\
& \quad
+ \sum_{\substack{p\leq x\\p\nmid N}}\sum_{0\leq j \leq m}
\bigg\{\log \left(1-\frac{\alpha_f(p)^{m-2j}}{p}\right)^{-1}
+ O\left(\frac{\log p}{x\log x}\right)\bigg\}.
\end{aligned}
\end{equation*}
Whence we obtain our result from \eqref{2.7} thanks to the prime number theorem.
\end{proof}

\vskip 8mm

\section{Proof of Theorem \ref{thm4}}

As in \cite[Theorem 1]{Lau_Wu}, we shall follow the method of Montgomery \cite{Montgomery1972}. 
First of all, we shall make a factorization of the symmetric power $L$-function. In the following, we fix a real parameter $z\geq 1$ (to be chosen explicitly later). 
We denote $P(z)=\prod_{p<z}p$.

\begin{lemma}\label{lemma 1}
Let $f\in \mathcal{H}^*_k(N), m\in \mathbb{N}$ and $z>(m+1)^2$. For $\sigma>1$, we have a factorization
\begin{equation*}
L(s,{\rm sym}^m f)^{-1}=G_f(s)L^{\flat}(s,{\rm sym}^m f)
\end{equation*}
with
\begin{equation*}
L^{\flat}(s,{\rm sym}^m f) := \sum\limits_{(n,P(z))=1}\lambda_{{\rm sym}^m f}(n)\mu(n)n^{-s},
\end{equation*}
where $G_f(s)$is holomorphic and has neither zeros nor poles in $\sigma>\tfrac{1}{2}$ and $G_f(s)\ll_{z,\varepsilon}1$ uniformly for $\sigma>\tfrac{1}{2}+\varepsilon$.
\end{lemma}
\begin{proof}
The proof is the same as Lemma 5.1 in \cite{Lau_Wu} and Lemma 9 in \cite{Kowalski_Michel}.
\end{proof}

The second lemma is a large sieve inequality on the Hecke eigenvalues in the level-weight aspects.
Similar results in level aspect and in weight aspect have been obtained by Duke and Kowalski \cite{Duke_Kowalski2000}
and by Lau and Wu \cite{Lau_Wu}, respectively.
Since the proof is rather similar, the only difference is to replace the convexity bound for 
$L(s, {\rm sym}^mf\times {\rm sym}^mg)$ in level aspect or in weight aspect by our convexity bound in level-weight aspect.
Thus we omit it.

\begin{lemma}\label{LargeSieve1}
Let $1\le m\le 4$, $L\geq 1$ and $\{a_\ell\}_{\ell\leq L}$ be a sequence of complex numbers. Then for any $\varepsilon>0$, we have
\begin{equation*}
\sum_{f\in\mathcal{H}^*_k(N)} \Big|\sum_{\ell\leq L}a_\ell\lambda_{{\rm sym}^m f}(\ell)\Big|^2 
\ll_{\varepsilon} (kN)^\varepsilon\left(L+(kN^2)^{D_m}L^{1/2+\varepsilon}\right)\sum_{\ell\leq L}|a_\ell|^2,
\end{equation*}
where $D_m=m(m+1)/4+1$ and the implied constant depends only on $\varepsilon$.
\end{lemma}

Now we are ready to count the number of zeros of symmetric $L$-function. First of all,
by \cite[Theorem 5.38]{Iwaniec_Kowalski}, we have
\begin{equation*}
N(\tfrac{1}{2},j,{\rm sym}^m f)-N(\tfrac{1}{2},j-1,{\rm sym}^m f)\ll \log(kNj).
\end{equation*}
So Theorem \ref{thm4} follows immediately if $T\geq (kN)^r$ for given $r>0$. We assume 
$$3\leq T \leq (kN)^r.$$

Cut $\alpha \leq \sigma \leq 1$ and $0\leq \tau \leq T$ into boxes of width $2\log^2(kN)$. 
There are at most $O(\log^3 (kN))$ zeros in each box $\alpha \leq \sigma \leq 1$ and $Y\leq \tau \leq Y+2\log^2(kN)$. 
Let $n_{{\rm sym}^m f}$ be the number of boxes which contain at least one zero $\rho$ of $L(s,{\rm sym}^m f)$. 
Then 
\begin{equation}\label{e 13}
N(\alpha,T,{\rm sym}^m f)\ll n_{{\rm sym}^m f}\log^3 (kN).
\end{equation}
So we only need to prove that 
\begin{equation*}
\sum_{f\in \mathcal{H}^*_k(N)} n_{{\rm sym}^m f}\ll_{r,\varepsilon}Tk^{E_{m,r}(1-\alpha)/(3-2\alpha)+\varepsilon}N^{E_{m,r}'(1-\alpha)/(3-2\alpha)+\varepsilon}.
\end{equation*}

Consider $\alpha \geq \frac{1}{2}+2\varepsilon$. Let $x,y\in [1,(kN)^{20m^2(1+r)}]$ and we define
$$M_x(s,{\rm sym}^m f)=G_f(s)\sum_{\substack{n\leq x\\(n,P(z))=1}}\lambda_{{\rm sym}^m f}(n)\mu(n)n^{-s},$$
where $G_f(s)$ and $P(z)$ are given in Lemma \ref{lemma 1}.

Let $\rho=\beta+{\rm i}\gamma$ with $\beta\geq \alpha \; (>\frac{1}{2}+\varepsilon)$ and $0\leq \gamma \leq T$ be a zero of $L(s,{\rm sym}^m f)$ and 
$\kappa=1/\log(kN), \,\kappa_1=1-\beta+\kappa \; (>0),\, \kappa_2=\frac{1}{2}-\beta+\varepsilon \; (<0)$.
Then
\begin{equation*}
\begin{aligned}
\text{e}^{-1/y}=\frac{1}{2\pi {\rm i}}&\int_{(\kappa_1)}\left(1-L(\rho+w,{\rm sym}^m f)M_x(\rho+w,{\rm sym}^m f)\right)\Gamma(w)y^w\,\dd{w}\\
&+\frac{1}{2\pi {\rm i}}\int_{(\kappa_1)}L(\rho+w,{\rm sym}^m f)M_x(\rho+w,{\rm sym}^m f)\Gamma(w)y^w\,\dd{w}.
\end{aligned}
\end{equation*}
The zero of $L(s,{\rm sym}^m f)$ cancels the pole of $\Gamma(w)$ at $w=0$. 
So we can shift the line of the integration of the second integral to the line $\re w=\kappa_2$. Then we have 
\begin{equation*}
\begin{aligned}\label{eqn-one}
\text{e}^{-1/y}
& = \frac{1}{2\pi \text{i}}\int_{(\kappa_1)}(1-L(\rho+w,{\rm sym}^m f))M_x(\rho+w, {\rm sym}^m ,f)\Gamma(w)y^w\,\dd{w}
\\
& \quad
+\frac{1}{2\pi \text{i}}\int_{(\kappa_2)}L(\rho+w,{\rm sym}^m f)M_x(\rho+w,{\rm sym}^m f)\Gamma(w)y^w\,\dd{w}.
\end{aligned}
\end{equation*}

For $\re w=\kappa_2=\frac{1}{2}-\beta+\varepsilon$, 
the convexity bound \eqref{ConvexityBound}, \eqref{eq 2} and Lemma \ref{lemma 1} imply
\begin{align*}
L(\rho+w,{\rm sym}^m f)
& \ll N^{m/4+\varepsilon}(k+T+|\im w|)^{(m+2)/4+\varepsilon},
\\
M_x(\rho+w,{\rm sym}^m f)
& \ll_\varepsilon x^{1/2+\varepsilon}.
\end{align*}
Thus, the contribution from $|\im w|\geq \log^2(kN)$ to the second integral of \eqref{eqn-one} is
\begin{equation*}
\begin{aligned}
&\ll x^{1/2+\varepsilon} y^{1/2-\alpha}  
\int_{|\im w| \geq \log^2(kN)}N^{m/4+\varepsilon}(k+T+|\im w|)^{(m+2)/4+\varepsilon}|\Gamma (w)||\,\dd{w}| 
\\
&\ll_\varepsilon x^{1/2+\varepsilon} y^{1/2-\alpha} N^{m/4+\varepsilon} (k+T)^{(m+2)/4+\varepsilon}{\rm e}^{-\log^2(kN)}
\ll_{\varepsilon,r} (kN)^{-1},
\end{aligned}
\end{equation*}
with $T\leq (kN)^r$.

According to \eqref{eq 2}, we have $L(s,{\rm sym}^m f)\leq \zeta(s)^{m+1}$ for $\re s>1$.
 So for $\re w=\kappa_1=1-\beta+\kappa$, it follows that
\begin{align*}
&1-L(\rho+w, {\rm sym}^m f)M_x(\rho+w, {\rm sym}^m f)
\\
& = L(\rho+w,{\rm sym}^m f)G_f(\rho+w)\sum_{\substack{n> x\\(n,P(z))=1}}
\frac{\mu(n)\lambda_{{\rm sym}^m f}(n)}{n^{-\rho-w}}
\ll (kN)^{\varepsilon}.
 \end{align*}
 Then contribution of $|\im w|\geq \log^2(kN)$ to the first integral of \eqref{eqn-one} is
 \begin{align*}
\ll_\varepsilon (kN)^\varepsilon y^{1-\alpha+\kappa} {\rm e}^{-\log^2(kN)}\ll_\varepsilon (kN)^{-1}.
 \end{align*}
Then using the fact that $1\leq C(a+b)\to 1\leq 2C^2(a^2+b)$ (where $a>0$, $b>0$ and $c\geq 1$),
we obtain
\begin{equation*}
\begin{aligned}
1 
& \ll_\varepsilon (kN)^\varepsilon y^{2(1-\alpha)}
\\
& \quad
\times \int_{-\log^2(kN)}^{\log^2(kN)} 
|1-L(1+\kappa +{\rm i}(\gamma+v),{\rm sym}^m f)M_x(1+\kappa+{\rm i}(\gamma+v),{\rm sym}^m f)|^2\,\dd{v}
\\
& \quad
+ y^{1/2-\alpha} \int_{-\log^2(kN)}^{\log^2(kN)} 
|L(\tfrac{1}{2}+\varepsilon+{\rm i}(\gamma+v),{\rm sym}^m f)M_x(\tfrac{1}{2}+\varepsilon+{\rm i}(\gamma+v),{\rm sym}^m f)|\,\dd{v}.
\end{aligned}
\end{equation*}
We separate the boxes into two groups, the odd-indexed and the even-indexed, then any two zeros from different boxes in the same group have a distance of at least $2\log^2(kN)$. Summing the integral over the zeros of these two groups separately, we obtain
\begin{equation}\label{e 10}
n_{{\rm sym}^m f}\ll (kN)^\varepsilon y^{2(1-\alpha)} I_1 + y^{1/2-\alpha} I_2,
\end{equation}
where
\begin{align*}
I_1
& := \int_0^{2T}|1-L(1+\kappa +{\rm i}v,{\rm sym}^m f)M_x(1+\kappa+{\rm i}v,{\rm sym}^m f)|^2\,\dd{v},
\\
I_2
& := \int_0^{2T}|L(\tfrac{1}{2}+\varepsilon +{\rm i}v,{\rm sym}^m f)M_x(\tfrac{1}{2}+\varepsilon+{\rm i}v,{\rm sym}^m f)|\,\dd{v}.
\end{align*}
For $T\leq (kN)^r$, we have
\begin{equation}\label{e 11}
\begin{aligned}
I_2 \ll_\varepsilon \int^{2T}_0 N^{m/4+\varepsilon}(k+v)^{(m+1)/4+\varepsilon}x^{1/2+\varepsilon}\,\dd{v}
\ll_\varepsilon Tx^{1/2+\varepsilon}k^{r(m+1)/4+r\varepsilon}N^{(mr+m+r)/4+r\varepsilon}.
\end{aligned}
\end{equation}
For $I_1$, we have
\begin{equation}\label{e 9}
\begin{aligned}
1-&L(1+\kappa +{\rm i}v,{\rm sym}^m f)M_x(1+\kappa+{\rm i}v,{\rm sym}^m f)\\
 \noalign{\vskip 2mm}
&\ll_\varepsilon(kN)^\varepsilon \Biggl|\sum_{\substack{x<n\leq X\\(n,P(z))=1}}\frac{\mu(n)\lambda_{{\rm sym}^m f}(n)}{n^{1+\kappa+{\rm i}v}}\Biggr|+(kN)^\varepsilon \sum_{n>X}\frac{d_{m+1}(n)}{n^{1+\kappa}},
\end{aligned}
\end{equation}
where $X={\rm e}^{4\log^2(kN)}$.

The second sum of \eqref{e 9} is $\ll (kN)^{-1}$.

With Lemma \ref{LargeSieve1}, the first sum in \eqref{e 9} is
\begin{equation*}
\sum_{f\in \mathcal{H}^*_k(N)}\biggl|\sum_{\substack{L<n\leq 2L\\(n,P(z))=1}}
\frac{\mu(n)\lambda_{{\rm sym}^m f}(n)}{n^{1+\kappa+{\rm i}v}}\biggr|^2 
\ll (kN)^\varepsilon \big(L+(kN^2)^{D_m}L^{1/2+\varepsilon}\big)L^{-1-2\kappa}.
\end{equation*}
Separating the range $x<n\leq X$ into dyadic intervals, we get by the Cauchy-Schwarz's inequality
$$
\sum_{f\in \mathcal{H}^*_k(N)}\biggl|\sum_{\substack{x<n\leq X\\(n,P(z))=1}} 
\frac{\mu(n)\lambda_{{\rm sym}^m f}(n)}{n^{1+\kappa+{\rm i}v}} \biggr|^2
\ll (kN^2)^{D_m+\varepsilon} x^{-1/2+\varepsilon}+1.
$$
Thus we have
\begin{equation}\label{e 12}
\begin{aligned}
\sum_{f\in \mathcal{H}^*_k(N)}I_1 
& \ll (kN)^\varepsilon \int_0^{2T}\sum_{f\in \mathcal{H}_k^*(N)}\biggl|\sum_{\substack{x<n\leq X\\(n,P(z))=1}}\frac{\mu(n)\lambda_{{\rm sym}^m f}(n)}{n^{1+\kappa+{\rm i}v}}\biggr|^2\,\dd{v}+T
\\
& \ll (kN^2)^{\varepsilon}T\left((kN^2)^{D_m}x^{-1/2+\varepsilon}+1\right).
\end{aligned}
\end{equation} 
Collecting \eqref{e 10}, \eqref{e 11} and \eqref{e 12}, we obtain
\begin{align*}
\sum_{f\in \mathcal{H}^*_k(N)}n_{{\rm sym}^m f}
& \ll_{r,\varepsilon} Tx^\varepsilon (kN)^{2r\varepsilon}
\\\noalign{\vskip -1mm}
& \quad
\times
\left[y^{2(1-\alpha)}\big(1+(kN^2)^{D_m}x^{-1/2}\big)+y^{1/2-\alpha} x^{1/2} k^{r(m+1)/4+1}N^{(mr+m+r)/4+1}\right].
\end{align*}
Taking $x=(kN^2)^{2D_m}$ and $y=k^{E_{m,r}/(2(3-2 \alpha ))}N^{E'_{m,r}/(2(3-2 \alpha ))}$,
we get
\begin{equation*}
\sum_{f\in {\mathcal{H}^*_k(N)}} n_{{\rm sym}^m f} 
\ll_{r,\varepsilon} Tk^{E_{m,r}(1-\alpha)/(3-2\alpha)+\varepsilon}N^{E_{m,r}'(1-\alpha)/(3-2\alpha)}.
\end{equation*}
It implies Theorem \ref{thm4} by \eqref{e 13}.

\vskip 8mm

\section{Complex moments of $L(1,{\rm sym}^m f)$}

The aim of this section is to compute the complex moments of $L(1,{\rm sym}^m f)$ in the level-weight aspect.

\subsection{Notations and statement of the result}\

\vskip 1mm

First we introduce some notations which are a bit heavy but carry interpretations in representation theory. 
The details can be found in \cite{Cogdell_Michel}.
For $\theta\in \mathbb{R}$, $m\in \mathbb{N}$, $|x|<1$ and $z\in \mathbb{C}$, we denote
\begin{equation}\label{eq 25}
\begin{aligned}
g(\theta)
& :={\rm diag}\big[{\rm e}^{{\rm i}\theta},{\rm e}^{-{\rm i}\theta}\big],
\\
{\rm sym}^m[g(\theta)]
& :={\rm diag}\big[{\rm e}^{{\rm i}m\theta},{\rm e}^{{\rm i}(m-2)\theta},\dots,{\rm e}^{-{\rm i}m\theta}\big],
\\
D\big(x,{\rm sym}^m[g(\theta)]\big)
& :=\det\big(I-x\cdotp {\rm sym}^m [g(\theta)]\big)^{-1}
= \prod_{0\leq j\leq m} \big(1-{\rm e}^{{\rm i}(m-2j)\theta}x\big)^{-1}.
\end{aligned}
\end{equation}
And for $z\in \mathbb{C}$, $m\in \mathbb{N}$ and $\nu\geq 0$, define $\lambda_m^{z,\nu}[g(\theta)]$ by
\begin{equation*}
D(x,{\rm sym}^m[g(\theta)])^z=\sum_{\nu\geq 0}\lambda_m^{z,\nu}[g(\theta)]x^{\nu},\qquad (|x|<1).
\end{equation*}
Then we have
\begin{equation}\label{eq 26}
\begin{aligned}
\lambda_m^{1,1}[g(\theta)]
& = {\rm tr}({\rm sym}^m[g(\theta)]) = \frac{\sin[(m+1)\theta]}{\sin \theta},
\\
\log D(x,{\rm sym}^m[g(\theta)])
& = {\rm tr}({\rm sym}^m[g(\theta)])x+O(x^2) 
\qquad (|x|<1).
\end{aligned}
\end{equation}
According to \eqref{eq 1}, for $p\nmid N$, we can denote $\alpha_f(p)={\rm e}^{{\rm i}\theta_f(p)}$ where $\theta_f(p)\in[0,\pi]$.
Then 
\begin{equation}\label{eq 33}
\lambda_f(p^m) = \frac{\sin[(m+1)\theta_f(p)]}{\sin\theta_f(p)}
= {\rm tr}\big({\rm sym}^m[g(\theta_f(p))]\big)
=\lambda_m^{1,1}[g(\theta_f(p))].
\end{equation}
According to \eqref{defLssymmf}, we have
\begin{equation*}
L(s,{\rm sym}^m f)^z=\prod_{p\mid N} \big(1-\varepsilon_f^m(p)p^{-(m/2+s)}\big)^{-z}
\prod_{p\nmid N}D\big(p^{-s}, {\rm sym}^m[g(\theta_f(p))]\big)^z,
\end{equation*}
and it admits a Dirichlet series
\begin{equation*}
L(s,{\rm sym}^m f)^z=\sum_{n\geq 1}\lambda_{{\rm sym}^m f}^z(n)n^{-s} 
\qquad (\sigma>1).
\end{equation*}
So $\lambda_{{\rm sym}^m f}^z(n)$ is multiplicative and we have
\begin{equation}\label{e 2}
\lambda_{{\rm sym}^m f}^z(p^\nu)=
\left\{
\begin{aligned}
&\lambda_m^{z,\nu}[g(\theta_f(p))] &&\mbox{if $p \nmid N$},
\\\noalign{\vskip 1mm}
&d_z(p^\nu)\lambda_f(p^{m\nu})  &&\mbox{if $p \mid N$},
\end{aligned}
\right.
\end{equation}
where $d_z(n)$ is a multiplicative function defined by 
$\sum_{n=1}^{\infty}d_z(n)n^{-s}=\zeta(s)^z$ for $\re s>1$.

We also define
\begin{equation}\label{defAmpmBmpm}
\left\{
\begin{aligned}
A_m^{\pm}
& :=\max_{\theta\in[0,\pi]}\pm {\rm tr}({\rm sym}^m[g(\theta)])=\pm{\rm tr}({\rm sym}^m[g(\theta_m^{\pm})]),
\\
B_m^{\pm}
& := \exp\Big\{\varpi_0+(A_m^{\pm})^{-1}\sum_p\big(\pm\log D(p^{-1},{\rm sym}^m[g(\theta_{m,p}^\pm)])-A_m^{\pm}p^{-1}\big)\Big\}.
\end{aligned}
\right.
\end{equation}
Here $\varpi_0$ is defined by
$\sum_{p\leq t} p^{-1} = \log_2 t+\varpi_0+O(\log^{-1} t)$ and $\theta_{m,p}^{\pm}\in[0,\pi]$ defined by
\begin{equation}\label{defthetamp}
\left\{
\begin{aligned}
& D\big(p^{-1},{\rm sym}^m[g(\theta^+_{m,p})]\big)=\max_{\theta\in[0,\pi]} D\big(p^{-1},{\rm sym}^m[g(\theta)]\big)
\\
& D\big(p^{-1},{\rm sym}^m[g(\theta^-_{m,p})]\big)=\min_{\theta\in[0,\pi]} D\big(p^{-1},{\rm sym}^m[g(\theta)]\big)
\end{aligned}
\right.
\end{equation}
are computed in \cite{Lau_Wu}.

For $n \in \mathbb{N}$, we write $n=n_Nn^{(N)}$ with $p\mid n_N \Rightarrow p\mid N$ 
and $(n_N,n^{(N)})=1$.
We define
\begin{equation}\label{defMsymmzN}
M^z_{{\rm sym}^m}(N):=\sum_{n\geq 1}\frac{\Box_N(n^m)d_z(n)}{n^{1+m/2}}
\prod_{p\nmid N}\frac{2}{\pi}\int_0^{\pi} D\big(p^{-1},{\rm sym}^m[g(\theta)]\big)^z\sin^2 \theta\,\dd{\theta},
\end{equation}
where $\Box_N(n)$ is defined by
$$
\sum_{n=1}^{\infty}\frac{\Box_N(n)}{n^s}:=\zeta_N(2s):=\prod_{p\mid N} \big(1-p^{-2s}\big)^{-1}.
$$
We also put
\begin{equation}\label{defMsymmz}
M^z_{{\rm sym}^m} := M^z_{{\rm sym}^m}(1)
= \prod_{p} \frac{2}{\pi}\int_0^{\pi} D\big(p^{-1},{\rm sym}^m[g(\theta)]\big)^z\sin^2 \theta\,\dd{\theta}.
\end{equation}

About the complex moments of $L(1,{\rm sym}^m f)$, we have the following result,
which will play a key role in the proof of Theorems \ref{thm1} and \ref{thm3}.

\begin{proposition}\label{prop4.1}
Let $\eta\in (0, \tfrac{1}{65})$ be fixed, $1\le m\le 4$, $2\mid k$ and $N$ be square free. Then there are two positive constants $\delta=\delta(\eta)$ and $c=c(\eta)$ such that
\begin{equation*}
\sum_{f\in \mathcal{H}^+_k(N;\eta,m)}\omega_f L(1,{\rm sym}^m f)^z=M^z_{{\rm sym}^m}(N)+O_{\eta}({\rm e}^{-\delta\log (kN)/\log_2(kN)})
\end{equation*}
uniformly for $|z|\leq c\log (kN)/\log_2(10kN)\log_3(10kN)$.
\end{proposition}

\subsection{Preliminary lemmas}\

\vskip 1mm

\begin{lemma}\label{lemma 4.1}
Let $2\mid k$ and $N$ be square free, $m\in \mathbb{N}$ and $z\in \mathbb{C}$. For $f\in \mathcal{H}^*_k(N)$, $p\nmid N$ and integer $\nu >0$, we have
\begin{equation}\label{e 27}
\lambda^z_{{\rm sym}^m f}(p^\nu)=\sum_{0\leq \nu' \leq m\nu}\mu_{m,\nu'}^{z,\nu}\lambda_f(p^{\nu'}),
\end{equation}
where
\begin{equation*}
\mu_{m,\nu'}^{z,\nu}
= \frac{2}{\pi}\int_0^{\pi}\lambda^{z,\nu}_m[g(\theta)]\sin[(\nu'+1)\theta]\sin\theta\,\dd{\theta}.
\end{equation*}
Further more, we have
\begin{equation}\label{e 1}
\begin{aligned}
\mu_{m,\nu'}^{z,1}&=z\delta(m,\nu') \qquad (0\leq\nu'\leq m),\\
|\mu_{m,\nu'}^{z,\nu}|&\leq d_{(m+1)|z|}(p^\nu) \qquad (0\leq \nu'\leq m\nu),\\
\sum_{0\leq \nu'\leq m\nu}|\mu_{m,\nu'}^{z,\nu}|&\leq d_{(m+1)|z|}(p^\nu),
\end{aligned}
\end{equation}
where $\delta(a,b)$ is $1$ for $a=b$ and $0$ otherwise.
\end{lemma}

\begin{proof}
The proof is the same as \cite[Lemma 6.1]{Lau_Wu} and \eqref{e 1} follows from \cite[Proposition 2]{Emmanuel_Jie}.
\end{proof}

\begin{lemma}
Let $2\mid k$ and $N$ be square free, $m,n\in \mathbb{N}$ and $z\in \mathbb{C}$. We have
\begin{equation}\label{eq 8}
\sum_{f\in \mathcal{H}^*_k(N)}\omega_f\lambda^z_{{\rm sym}^m f}(n)
=\lambda_{{\rm sym}^m}^z (n)+O_m\left(k ^{-5/6}N^{-1+\varepsilon}n^{m/4}\log (2n)r^z_m(n)\right),
\end{equation}
where $\lambda_{{\rm sym}^m}^z (n)$ is the multiplicative function defined by
\begin{equation*}
\lambda_{{\rm sym}^m}^z (p^\nu):=\left\{
\begin{aligned}
&\mu_{m,0}^{z,\nu} &&\mbox{if $p \nmid N$},
\\
& d_z(p^\nu)\Box(p^{m\nu})/\sqrt{p^{m\nu}}  &&\mbox{if $p \mid N$}.
\end{aligned}
\right.
\end{equation*}
Here
$\Box(n)=1$ if $n$ is a square, and $\Box(n)=0$ otherwise, and $r^z_m(n)$ is the multiplicative function defined by
\begin{equation}\label{e 7}
r^z_m(p^\nu):=\left\{
\begin{aligned}
&d_{(m+1)|z|}(p^\nu) &&\mbox{if $p \nmid N$},
\\
& d_{|z|}(p^\nu)/p^{m\nu/2} &&\mbox{if $p \mid N$}.
\end{aligned}
\right.
\end{equation}
\end{lemma}

\begin{proof}
Write $n=q_1^{\overline{\nu}_1}\cdots q_h^{\overline{\nu}_h}p_1^{\nu_1}\cdots p_r^{\nu_r}$ where $q_i\mid N$ for $1\leq i\leq h$ and $p_j\nmid N$ for $1\leq j\leq r$. We have according to 
\eqref{e 27}
\begin{equation*}
\begin{aligned}
\sum_{f\in \mathcal{H}^*_k(N)}\omega_f\lambda^z_{{\rm sym}^m f}(n)=d_z(q_1^{\overline{\nu}_1}\cdots q_h^{\overline{\nu}_h})&\sum_{\nu'_1=0}^{m\nu_1}\cdots \sum_{\nu'_r=0}^{m\nu_r}\left(\prod_{j=1}^r\mu_{m,\nu_j'}^{z,\nu_j}\right)\\
&\times\sum_{f\in \mathcal{H}^*_k(N)}\omega_f\lambda_f(q_1^{m\overline{\nu}_1}\cdots q_h^{m\overline{\nu}_h}p_1^{\nu_1'}\cdots p_r^{\nu_r'}).
\end{aligned}
\end{equation*}
If we write $q_1^{m\overline{\nu}_1}\cdots q_h^{m\overline{\nu}_h}=g^2q$, according to \eqref{eq 1} and using the trace formula Corollary 2.10 in \cite{Iwaniec_Luo_Sarnak}, we get the main term $\lambda_{{\rm sym}^m}^z (n)$, and the error term is
\begin{equation*}
\begin{aligned}
\ll&\sum_{\nu_1'}^{m\nu_1}\cdots\sum_{\nu_r'}^{m\nu_r}\biggl(\prod_{j=1}^r\nu_{m,\nu_j'}^{z,\nu_j}\biggr)\frac{d_z(n_N)(qp_1^{\nu_1'}\cdots p_r^{\nu_r'})^{1/4}\tau^2(N)\log(2qp_1^{\nu_1'}\cdots p_r^{\nu_r'}N)}{gk^{5/6}q^{1/2}\varphi(N)}  \\
\ll &N^{-1+\varepsilon}k^{-5/6}n^{m/4}\log(2n)\frac{d_z(q_1^{\overline{\nu}_1}\cdots q_h^{\overline{\nu}_h})}{gq^{1/2}}      \prod_{j=1}^r\sum_{\nu'_j=0}^{m\nu_j}|\mu_{m,\nu_j'}^{z,\nu_j}|,
\end{aligned}
\end{equation*}
which implies \eqref{eq 8} immediately by \eqref{e 1}.
\end{proof}

We define
\begin{equation*}
\omega^z_{{\rm sym}^m f}(x):=\sum_{n=1}^{\infty}\frac{\lambda^z_{{\rm sym}^m f}(n)}{n}{\rm e}^{-n/x}.
\end{equation*}

\begin{lemma}\label{eq 21}
Let $2\mid k$, $N$ be square free, $m\in \mathbb{N}$, $x\geq 3$ and $z\in \mathbb{C}$. 
For any $\varepsilon>0$, we have
\begin{equation*}
\sum_{f\in \mathcal{H}^*_k(N)}\omega_f \omega^z_{{\rm sym}^m f}(x)=\sum_{n=1}^{\infty}\frac{\lambda^z_{{\rm sym}^m}(n)}{n}{\rm e}^{-n/x}+O_m\left(k^{-5/6}N^{-1+\varepsilon}x^{m/4}[(z_m+1)\log x]^{z_m}\right),
\end{equation*}
where $z_m=(m+1)|z|+1$.
\end{lemma}

\begin{proof}
By the definition of $\omega^z_{{\rm sym}^m f}(x)$ and \eqref{eq 8}, we have
\begin{align*}
\sum_{f\in \mathcal{H}^*_k(N)} \omega_f \omega^z_{{\rm sym}^m f}(x)
& = \sum_{n=1}^{\infty}\frac{{\rm e}^{-n/x}}{n}\sum_{f\in \mathcal{H}^*_k(N)}\omega_f\lambda^z_{{\rm sym}^m f}(n)
\\
& = \sum_{n=1}^{\infty} \frac{\lambda^z_{{\rm sym}^m}(n)}{n}{\rm e}^{-n/x}
+ O\Big(k^{-5/6}N^{-1+\varepsilon}\sum_{n=1}^{\infty}n^{m/4-1}\log(2n){\rm e}^{-n/x}r_m^z(n)\Big).
\end{align*} 
According to \eqref{e 7}, we have $r_m^z(n)\leq d_{(m+1)|z|}(n)$. And one has the property of $d_l(n)$,
\begin{equation*}
\sum_{n\leq X}\frac{d_{\ell}(n)}{n}\leq \bigg(\sum_{n\leq X}\frac{1}{n}\bigg)^{\ell}\leq (\log 3X)^{\ell}.
\end{equation*}
Thus the sum in the error term is
\begin{equation*}
\begin{aligned}
= \int^{\infty}_{1}\log(2t)t^{m/4}{\rm e}^{-t/x}\,\dd\Big(\sum_{n\leq t} r_m^z(n)n^{-1}\Big)
                            \ll_m x^{m/4}[(z_m+1)\log x]^{z_m}.
\end{aligned}
\end{equation*}
This completes the proof.
\end{proof}

The proof of the following lemma can be found in \cite{Lau_Wu}.

\begin{lemma}\label{e 3}
Let $m\in \mathbb{N}$, $z\in \mathbb{C}$ and define $z_m':=(m+1)|z|+3$. Then there exits a constant $c=c(m)>0$ such that
\begin{equation*}
\sum_{(n,N)=1} \frac{|\lambda^z_{{\rm sym}^m}(n)|}{n^{\sigma}}
\leq \exp\bigg\{cz_m'\bigg(\log_2 z_m'+\frac{{z_m'}^{(1-\sigma)/\sigma}-1}{(1-\sigma)\log z_m'}\bigg)\bigg\}
\end{equation*}
for any $\sigma\in (\tfrac{1}{2},1]$.
Further we have
\begin{equation*}
\sum_{(n,N)=1}\frac{\lambda_{{\rm sym}^m}^z(n)}{n}
= \prod_{p\nmid N}\frac{2}{\pi}\int_0^{\pi}D\big(p^{-1},{\rm sym}^m[g(\theta)]\big)^z\sin^2\theta\,\dd{\theta}.
\end{equation*}
\end{lemma}

\begin{lemma}\label{eq 22}
Let $m\in \mathbb{N}$, $\sigma\in [0, 1/3)$, $x\geq 3$ and $z\in \mathbb{C}$. There exists a constant $c=c(m)$ such that
\begin{equation*}
\sum_{n=1}^\infty\frac{\lambda_{{\rm sym}^m}^z(n)}{n}{\rm e}^{-n/x}
=M^z_{{\rm sym}^m}(N)
+O\bigg(x^{-\sigma}\exp\bigg\{cz_m'\biggl(\log_2 z_m'+\frac{{z_m'}^{\sigma/(1-\sigma)}-1}{\sigma\log z_m'}\biggr)\bigg\}\bigg).
\end{equation*}
The implied constant depends on $m$ only.
\end{lemma}

\begin{proof}
According to the definition of $\lambda_{{\rm sym}^m}^z(n)$, write $n=n_Nn^{(N)}$,
 where $n_N \mid N^{\infty}$ and $(n^{(N)},N)=1$, then we have
\begin{equation*}
\begin{aligned}
\sum_{n=1}^\infty \frac{\lambda_{{\rm sym}^m}^z(n)}{n}{\rm e}^{-n/x}
& = \sum_{n_N=1}^{\infty}\frac{\lambda^z_{{\rm sym}^m}(n_N)}{n_N}\sum_{(n^{(N)},N)=1}\frac{\lambda^z_{{\rm sym}^m}(n^{(N)})}{n^{(N)}}{\rm e}^{-n_Nn^{(N)}/x}
\\
& = \sum_{n=1}^\infty\frac{d_z(n)\Box_N(n^m)}{n^{m/2+1}}\sum_{(n^{(N)},N)=1} 
\frac{\lambda^z_{{\rm sym}^m}(n^{(N)})}{n^{(N)}}{\rm e}^{-n_Nn^{(N)}/x}.
\end{aligned}
\end{equation*}
We write
\begin{equation*}
\sum_{(n^{(N)},N)=1}\frac{\lambda^z_{{\rm sym}^m}(n^{(N)})}{n^{(N)}}{\rm e}^{-n_Nn^{(N)}/x}=\sum_{(n^{(N)},N)=1}\frac{\lambda^z_{{\rm sym}^m}(n^{(N)})}{n^{(N)}}+O(R_1+R_2),
\end{equation*}
where
\begin{equation*}
R_1:=\sum_{\substack{(n^{(N)},N)=1\\n^{(N)}>x/n_N}}\frac{|\lambda^z_{{\rm sym}^m}(n^{(N)})|}{n^{(N)}},
\qquad 
R_2:=\sum_{\substack{(n^{(N)},N)=1\\n^{(N)}\leq x/n_N}} \frac{|\lambda^z_{{\rm sym}^m}(n^{(N)})|}{n^{(N)}}
\big|{\rm e}^{-n_Nn^{(N)}/x}-1\big|.
\end{equation*}
For any $\sigma \in [0, \tfrac{1}{3})$, we have
\begin{equation*}
(n/x)^\sigma \gg \left\{
\begin{aligned}
&1 &&\mbox{if $n>x$},\\
&|{\rm e}^{-n/x}-1|  &&\mbox{if $n\leq x$}.
\end{aligned}
\right.
\end{equation*}
So by Lemma \ref{e 3}, we have
\begin{equation*}
\begin{aligned}
R_1+R_2 
& \ll \sum_{(n^{(N)},N)=1}\frac{|\lambda^z_{{\rm sym}^m}(n^{(N)})|}{n^{(N)}}\left(\frac{n_Nn^{(N)}}{x}\right)^\sigma
\\
& \ll \left(\frac{n_N}{x}\right)^{\sigma}
\exp\bigg\{cz_m'\bigg(\log_2 z_m'+\frac{{z_m'}^{(\sigma/1-\sigma)}-1}{\sigma\log z_m'}\bigg)\bigg\},
\end{aligned}
\end{equation*}
and
\begin{equation}\label{e 4}
\begin{aligned}
\sum_{n=1}^\infty \frac{\lambda_{{\rm sym}^m}^z(n)}{n}{\rm e}^{-n/x}
& = \sum_{n=1}^\infty \frac{d_z(n)\Box_N(n^m)}{n^{m/2+1}} 
\Bigg[\prod_{p\nmid N} \frac{2}{\pi} \int_0^{\pi} D\big(p^{-1},{\rm sym}^m[g(\theta)]\big)^z\sin^2\theta\,\dd{\theta}
\\
& \quad
+ O\bigg(\Big(\frac{n_N}{x}\Big)^{\sigma} 
\exp\bigg\{cz_m'\bigg(\log_2 z_m'+\frac{{z_m'}^{\sigma/(1-\sigma)}-1}{\sigma\log z_m'}\bigg)\bigg\}\bigg)\Bigg].
\end{aligned}
\end{equation}
According to the definition of $d_z(n)$, we have
\begin{equation*}
\begin{aligned}
\sum_{n=1}^\infty \frac{d_z(n)\Box_N(n^m)}{n^{m/2+1-\sigma}}
\ll \sum_{n=1}^\infty\frac{d_z(n)}{n^{m/2+1-\sigma}}\ll {\rm e}^{c|z|}.
\end{aligned}
\end{equation*}
We complete the proof by inserting it back to \eqref{e 4}.
\end{proof}

\begin{lemma}\label{eq 19}
Let $\eta \in (0, \tfrac{1}{65})$ fixed, $1\le m\le 4$, $2\mid k$, $N$ be square free and $f\in \mathcal{H}^+_k(N;\eta,m)$. 
Then we have
\begin{equation*}
L(1,{\rm sym}^m f)^z
=\omega^z_{{\rm sym}^m f}(x)
+ O\Big(\big(x^{-1/\log_2 (kN)}+x^{c|z|}{\rm e}^{-\log^2(kN)}\big){\rm e}^{c|z|\log_3 (10kN)}\Big)
\end{equation*}
uniformly for $x>3$ and $z\in \mathbb{C}$, where the constant $c=c(\eta)$ and the implied constant depends on $\eta$ only.
\end{lemma}

\begin{proof}
We begin our proof with the equation
\begin{equation*}
\omega^z_{{\rm sym}^m f}(x)=\frac{1}{2\pi{\rm i}}\int_{(1)}L(s+1, {\rm sym}^m f)^z\Gamma(s)x^s\,\dd{s}.
\end{equation*}
Move the integral to the path $\mathcal{C}$ consisting of the straight lines joining
$$\kappa_1-{\rm i}\infty,\quad \kappa_1-{\rm i}T,\quad -\kappa_2-{\rm i}T,\quad -\kappa_2+{\rm i}T,\quad \kappa_1+{\rm i}T, \quad \kappa_1+{\rm i}\infty,$$
where $\kappa_1=1/\log x$, $\kappa_2=1/\log_2 (kN)$ and $T=\log^2(kN)$. Then we have
$$\omega^z_{{\rm sym}^m f}(x)=L(1,{\rm sym}^m f)^z+\frac{1}{2\pi {\rm i}}\int_{\mathcal{C}}L(s+1, {\rm sym}^m f)^z\Gamma(s)x^s\dd{s}.$$
By \eqref{eq 17} and Proposition \ref{lem2.5} we get
\begin{equation*}
\begin{aligned}
\frac{1}{2\pi{\rm i}}\int_{\mathcal{C}}L(s&+1, {\rm sym}^m f)^z\Gamma(s)x^s\,\dd{s}
\ll_{\eta}\, x^{-\kappa_2}{\rm e}^{c|z|\log_3 (10kN)}\int_{|y|\leq T}|\Gamma(1-\kappa_2+{\rm i}y)|\,\dd{y}+\\
&+{\rm e}^{c|z|\log_3 (10kN)}\int_{-\kappa_2}^{\kappa_1}|\Gamma(1+\alpha+{\rm i}T)|\,\dd{\alpha}
+{\rm e}^{c|z|\log x}\int_{|y|\geq T}|\Gamma(1+\kappa_1+{\rm i}y)|\,\dd{y}\\
&\ll \left(x^{-1/\log_2 (kN)}+x^{c|z|}{\rm e}^{-\log^2(kN)}\right){\rm e}^{c|z|\log_3 (10kN)},
\end{aligned}
\end{equation*}
by Stirling formula.
\end{proof}

\subsection{Proof of Proposition \ref{prop4.1}}\

\vskip 1mm

By Lemma \ref{eq 19}, we have
\begin{equation}\label{eq 20}
\sum_{f\in \mathcal{H}^+_k(N;\eta,m)}\omega_f L(1,{\rm sym}^m f)^z=\sum_{f\in \mathcal{H}^+_k(N; \eta, m)}\omega_f \omega^z_{{\rm sym}^m f}(x)+O_{\eta}(R_1),
\end{equation}
where
$$R_1=\sum\limits_{f\in \mathcal{H}^+_k(N; \eta, m)}\omega_f(x^{-1/\log_2 (kN)}+x^{c|z|}{\rm e}^{-\log^2(kN)}){\rm e}^{c|z|\log_3(10kN)}.$$
Then with the trace formula \cite[Corollary 2.10]{Iwaniec_Luo_Sarnak}, we have
$$R_1\ll (x^{-1/\log_2 (kN)}+x^{c|z|}{\rm e}^{-\log^2(kN)}){\rm e}^{c|z|\log_3(10kN)}.$$
For $\varepsilon>0$ which is a constant given later and $f\in \mathcal{H}_k^*(N)$ we have
$$\omega^z_{{\rm sym}^m f}(x)=\frac{1}{2\pi {\rm i}}\int_{(\varepsilon)}L(s+1,{\rm sym}^m f)^z\Gamma(s)x^s\,\dd{s}\ll \zeta(1+\varepsilon)^{(m+1)|\re z|}x^{\varepsilon}.$$
Then considering the summation through $\mathcal{H}^-_k(N; \eta, m)$ and with the bound
$$(\log (kN))^{-1}\leq L(1,{\rm sym}^2 f)\leq \log (kN),$$ 
we get
$$\sum_{f\in \mathcal{H}^-_k(N; \eta, m)}\omega_f \omega^z_{{\rm sym}^m f}(x)\ll_{\eta} \zeta(1+\varepsilon)^{(m+1)|\re z|}x^{\varepsilon}(kN)^{65\eta-1}.$$
Together with \eqref{eq 20}, we have
$$\sum_{f\in \mathcal{H}^+_k(N;\eta,m)}\omega_f L(1,{\rm sym}^m f)^z=\sum_{f\in \mathcal{H}^*_k(N)} \omega_f\omega^z_{{\rm sym}^m f}(x)+O_{\eta}(R_2),$$
where
$$R_2=R_1+\zeta(1+\varepsilon)^{(m+1)|\re z|}x^{\varepsilon}(kN)^{65\eta-1}.$$
With Lemmas \ref{eq 21} and \ref{eq 22}, we get
\begin{equation*}
\sum_{f\in \mathcal{H}^+_k(N;\eta,m)}\omega_f L(1,{\rm sym}^m f)^z=M^z_{{\rm sym}^m}(N)+O_{\eta}(R_3),
\end{equation*}
where
\begin{equation*}
\begin{aligned}
R_3
& = x^{-\sigma}\exp\biggl\{cz_m'\biggl(\log_2 z_m'+\frac{{z_m'}^{\sigma/(1-\sigma)}-1}{\sigma\log z_m'}\biggr)\biggr\}
+k^{-5/6}N^{-1+\varepsilon}x^{m/4}\left[(z_m+1)\log x\right]^{z_m}
\\
\noalign{\vskip 3mm}
& \quad
+ \big(x^{-1/\log_2 (kN)}+x^{c|z|}{\rm e}^{-\log^2(kN)}\big){\rm e}^{c|z|\log_3(10kN)}+\zeta(1+\varepsilon)^{(m+1)|\re z|}x^{\varepsilon}(kN)^{65\eta-1}.
\end{aligned}
\end{equation*}
Taking 
$\varepsilon=\tfrac{1}{500m}$,
$x=(kN)^{\tfrac{1}{10m}}$
and $\sigma=\tfrac{1}{\log(|z|+8)}$,
we get positive constants $c_0$ and $\delta$ depending on $\eta$ such that
$$R_3\ll {\rm e}^{-\delta \log (kN)/\log_2 (kN)},$$
uniformly for $|z|\ll c_0\log (kN)/\log_2(10kN)\log_3(10kN)$.

\vskip 8mm

\section{Proof of Theorem \ref{thm1}}

\subsection{Proof of Theorem \ref{thm1}(i)}\

\vskip 1mm

In Lemma \ref{lem2.4}, by taking $s=1$ and $T=\log^{4/\eta} (kN)$, we can get
$$\log L(1,{\rm sym}^m f)=\sum_{n=1}^{\infty}\frac{\Lambda_{{\rm sym}^m f}(n)}{n\log n}{\rm e}^{-n/T}+O_{\eta}(\log^{-1} (kN)).$$
According to Lebesgue's dominated convergence theorem, we have
\begin{equation*}
\sum_p\sum_{\nu\geq 2}\frac{\Lambda_{{\rm sym}^m f}(p^{\nu})}{p^\nu \log p^{\nu}}
\big({\rm e}^{-p^\nu/T}-{\rm e}^{-\nu p/T}\big)\to 0.
\end{equation*}
for $kN\to \infty$ with $2\mid k$ and $N\in \mathbb{N}_k(\Xi)$.
So we get
\begin{equation*}
\begin{aligned}
\sum_{n=1}^\infty\frac{\Lambda_{{\rm sym}^m f}(n)}{n\log n}{\rm e}^{-n/T}
           &=\sum_{p}\sum_{\nu\geq 1}\frac{\Lambda_{{\rm sym}^m f}(p^{\nu})}{p^\nu \log p^{\nu}}{\rm e}^{-p^\nu/T}\\
           &=\sum_{p}\sum_{\nu\geq 1}\frac{\Lambda_{{\rm sym}^m f}(p^{\nu})}{p^\nu \log p^{\nu}}{\rm e}^{-\nu p/T}+o(1).
\end{aligned}
\end{equation*}
Since $P^-(N)\ge \log(kN)\log_2(kN)\to \infty$ as $kN\to \infty$, we have
\begin{equation*}
\sum_{p\mid N}\sum_{\nu\geq 1}\frac{\Lambda_{{\rm sym}^m f}(p^{\nu})}{p^\nu \log p^{\nu}}{\rm e}^{-\nu p/T}=o(1) 
\qquad 
(kN\to \infty).
\end{equation*}
Therefore we obtain
\begin{align*}
\sum_{n=1}^\infty\frac{\Lambda_{{\rm sym}^m f}(n)}{n\log n}{\rm e}^{-n/T}
& = \sum_{p\nmid N}\sum_{\nu\geq 1}\sum_{0\leq j\leq m}\frac{\alpha_f(p)^{(m-2j)\nu}}{\nu p^\nu}{\rm e}^{-\nu p/T}+o(1)
\\
& = \sum_{p\nmid N}\sum_{0\leq j\leq m}\log\left(1-\frac{\alpha_f(p)^{m-2j}}{{\rm e}^{p/T}p}\right)^{-1}+o(1)
\\
& = \sum_{p\nmid N}\log D\left({\rm e}^{-p/T}p^{-1}, {\rm sym}^m[g(\theta_f(p))]\right)+o(1),
\end{align*}
according to \eqref{eq 25} with $\theta\in[0,\pi]$ and $\alpha_f(p)={\rm e}^{{\rm i}\theta_f(p)}$.
By \eqref{eq 26} and \eqref{defAmpmBmpm}, we have
\begin{equation*}
\Big|\sum_{\substack{p> T\\ (p,N)=1}}
\log D\big({\rm e}^{-p/T}p^{-1}, {\rm sym}^m[g(\theta_f(p))]\big)\Big|
\ll \sum_{p>T} {\rm e}^{-p/T} p^{-1}
\ll (\log T)^{-1}
\to 0,
\end{equation*}
and
\begin{equation*}
\bigg|\sum_{\substack{p\leq T\\ (p,N)=1}}
\log\bigg(\frac{D({\rm e}^{-p/T}p^{-1}, {\rm sym}^m[g(\theta_f(p))])}{D(p^{-1}, {\rm sym}^m[g(\theta_f(p))])}\bigg)\bigg|
\ll \sum_{p\leq T} \frac{1-{\rm e}^{-p/T}}{p}
\ll (\log T)^{-1}\to 0.
\end{equation*}
So we get
\begin{equation}\label{eq 32}
\log L(1,{\rm sym}^m f)=\sum_{\substack{p\leq T\\\,(p,N)=1}}\log D(p^{-1},{\rm sym}^m[g(\theta_f(p))])+o(1).
\end{equation}
From \eqref{eq 32} and with the notation \eqref{defthetamp}, we have
\begin{equation}\label{eq 29}
\begin{aligned}
\sum_{\substack{p\leq T\\\,(p,N)=1}}\log D(p^{-1},{\rm sym}^m[g(\theta^+_{m,p})])+o(1)&\geq \log L(1,{\rm sym}^m f)\\
&\geq \sum_{\substack{p\leq T\\\,(p,N)=1}}\log D(p^{-1},{\rm sym}^m[g(\theta^-_{m,p})])+o(1).
\end{aligned}
\end{equation}
For one hand, from \eqref{eq 26} and \eqref{defAmpmBmpm}, we get
\begin{equation*}
\begin{aligned}
0\leq \mp\log\bigg(\frac{D(p^{-1},{\rm sym}^m[g(\theta^\pm_{m})])}{D(p^{-1},{\rm sym}^m[g(\theta^\pm_{m,p})])}\bigg)
& = \mp \frac{\pm A_m^{\pm}-{\rm tr}({\rm sym}^m[g(\theta^\pm_{m,p})])}{p}+O(p^{-2})
\\
&=- \frac{ A_m^{\pm}\mp{\rm tr}({\rm sym}^m[g(\theta^\pm_{m,p})])}{p}+O(p^{-2}).
\end{aligned}
\end{equation*}
For the other hand, $ A_m^{\pm}\mp{\rm tr}({\rm sym}^m[g(\theta^\pm_{m,p})])\geq 0$, we have
\begin{equation*}
\big(A_m^{\pm}\mp{\rm tr}({\rm sym}^m[g(\theta^\pm_{m,p})])\big) p^{-1}\ll p^{-2}.
\end{equation*}
Together with
\begin{equation*}
\log D(p^{-1},{\rm sym}^m[g(\theta^{\pm}_{m,p})])-{\rm tr}({\rm sym}^m[g(\theta^{\pm}_{m,p})])p^{-1}\ll p^{-2}
\end{equation*}
by \eqref{eq 26} again, we obtain
\begin{equation}\label{e 3.4}
\pm \log D(p^{-1},{\rm sym}^m[g(\theta^{\pm}_{m,p})])-A_m^{\pm}p^{-1}\ll p^{-2}.
\end{equation}
Therefore
\begin{equation*}
\sum_{\substack{p> T\\\,(p,N)=1}}\left(\pm \log D(p^{-1},{\rm sym}^m[g(\theta^{\pm}_{m,p})])-A_m^{\pm}p^{-1}\right)
\ll (T\log T)^{-1}.
\end{equation*}
Then we see
\begin{equation}\label{e 3.5}
\begin{aligned}
\sum_{\substack{p\leq T\\ (p,N)=1}} \log D\left(p^{-1},{\rm sym}^m[g(\theta^{\pm}_{m,p})]\right)
& \lesseqgtr \pm \sum_p\left(\pm\log D(p^{-1},{\rm sym}^m[g(\theta^{\pm}_{m,p})])-A_m^{\pm}p^{-1}\right)
\\
& \hskip -12mm
\pm \sum_{p\leq T} A_m^{\pm}p^{-1}
+ O\Big(\sum_{p\geq P^-(N)} p^{-2}+(T\log T)^{-1}+\sum_{\substack{p\leq T\\ p \mid N}} p^{-1}\Big).
\end{aligned}
\end{equation}
Since $N\in \mathbb{N}_k(\Xi)$ and by \eqref{defAmpmBmpm}, we have
\begin{equation*}
\sum_{\substack{p\leq T\\ (p,T)=1}}\log D\left(p^{-1},{\rm sym}^m[g(\theta^{\pm}_{m,p})]\right)\lesseqgtr\pm A^{\pm}_m\log \left(B^{\pm}_m \log T\right)+o(1).
\end{equation*}
 Put it back to \eqref{eq 29}, then we get \eqref{eq 30?}.
If GRH holds, we choose $s=1,\alpha=\tfrac{3}{4}$ and $T=(\log (kN))^{2+20\varepsilon}$, and with the same method we can get \eqref{Thm1.EqA}.

\subsection{Proof of Theorem \ref{thm1}(ii)}\

\vskip 1mm

We use Proposition \ref{prop4.1} to prove Theorem \ref{thm1}(ii).
Thanks to this proposition, for sufficiently large $kN$ with $2\mid k$ and $N\in \mathbb{N}_k(\Xi)$
and $r\le c\log (kN)/\log_2(10kN)\log_3(10kN)$, we have
\begin{equation*}
\sum_{f\in \mathcal{H}^+_k(N;\eta,m)}\omega_f L(1,{\rm sym}^m f)^{\pm r}\geq \frac{1}{2}M^{\pm r}_{{\rm sym}^m}(N).
\end{equation*}
Since
$$\sum_{f\in \mathcal{H}^+_k(N;\eta,m)}\omega_f \leq \sum_{f\in \mathcal{H}_k^*(N)} \omega_f=1+O(k^{-5/6}N^{-1+\varepsilon}),$$
there exist $f^{\pm}_m\in \mathcal{H}^+_k(N;\eta,m)$ such that
\begin{equation*}
L(1,{\rm sym}^m f_m^{\pm})^{\pm r}\geq \tfrac{1}{2}M^{\pm r}_{{\rm sym}^m}(N).
\end{equation*}

\begin{lemma}
For $N\in \mathbb{N}_k(\Xi)$ and $r\le c\log (kN)/\log_2(10kN)\log_3(10kN)$, we have
\begin{equation}\label{N1}
M^{\pm r}_{{\rm sym}^m}(N)=M^{\pm r}_{{\rm sym}^m}\exp\{O(r/\log^3 r)\}.
\end{equation}
\end{lemma}

\begin{proof}
According to the definition of $M^{\pm r}_{{\rm sym}^m}(N)$ as \eqref{defMsymmzN}, we have
\begin{equation*}
M^{\pm r}_{{\rm sym}^m}(N)
= M^{\pm r}_{{\rm sym}^m}
\sum_{n\geq 1} \frac{\Box_N(n^m)d_r(n)}{n^{1+m/2}}
\bigg(\prod_{p\mid N}\frac{2}{\pi}\int_0^{\pi}D(p^{-1},{\rm sym}^m[g(\theta)])^r\sin^2 \theta\,\dd{\theta}\bigg)^{-1}.
\end{equation*}
By the definitions of $\Box_N(\cdot)$ and $d_r(\cdot)$, we get
\begin{equation*}
\sum_{n\geq 1}\frac{\Box_N(n^m)d_r(n)}{n^{1+m/2}}=\prod_{p\mid N}\left(1-\frac{\Box(p^m)}{p^{m/2+1}}\right)^{-r}
= \exp\bigg\{O\bigg(\sum_{p\mid N}\frac{r}{p^{m/2+1}}\bigg)\bigg\}.
\end{equation*} 
Thanks to Lemma \ref{lemma 4.1}, we can obtain
\begin{equation*}
\frac{2}{\pi}\int_0^{\pi}D(p^{-1},{\rm sym}^m[g(\theta)])^r\sin^2 \theta\,\dd{\theta}
= \sum_{\nu\geq 0}\frac{\mu_{m,0}^{r,\nu}}{p^\nu}
=1+O\bigg(\frac{\mu_{m,0}^{r,2}}{p^2}\bigg).
\end{equation*}
Since $|\mu_{m,0}^{r,2}|\ll r^2$, we get
\begin{equation*}
M^{\pm r}_{{\rm sym}^m}(N)
= M^{\pm r}_{{\rm sym}^m} \exp\bigg\{O\biggl(\sum_{p\mid N}\frac{r}{p^{m/2+1}}+\frac{r^2}{p^2}\biggr)\bigg\}
\end{equation*}
So when $N\in \mathbb{N}_k(\Xi)$, the $O$ term follows.
\end{proof}

According to \cite{Emmanuel_Jie}, we have
\begin{equation}\label{e 18}
\log M^{\pm r}_{{\rm sym}^m}
= A^{\pm}_{m}r\log (B^{\pm}_m\log(A^{\pm}_{m}r)) + O\Big(\frac{r}{\log r}\Big),
\end{equation}
where $A^{\pm}_m$ and $B^{\pm}_m$ are positive constants defined as in \eqref{defAmpmBmpm} above.

By taking $r=c\log (kN)/\log_2(10kN)\log_3(10kN)$, we get \eqref{thm1.EqB}.

\vskip 8mm

\section{Large sieve inequality and Proof of Theorem \ref{thm2}}

\subsection{Large sieve inequality and application}\

\vskip 1mm

The following large sieve inequality is due to Lau and Wu \cite[Theorem 1]{Lau_Wu2},
which will play a key role in our proof of Theorem \ref{thm2}. 

\begin{lemma}\label{LargeSieve2}
Let $\nu\geq 1$ be a fixed integer. We have
\begin{equation*}
\sum_{f\in \mathcal{H}_k^*(N)} \bigg|\sum_{\substack{P<p\leq Q\\p\nmid N}} \frac{\lambda_f(p^\nu)}{p}\bigg|^{2j}
\ll_\nu k\varphi(N)\left(\frac{96(\nu+1)^2j}{P\log P}\right)^j+(kN)^{10/11}\left(\frac{10Q^{\nu/10}}{\log P}\right)^{2j}
\end{equation*}
uniformly for
\begin{equation*}
j\geq 1,
\qquad
2\mid k,
\qquad
N\;\;(\text{square free}),
\qquad
2\leq P<Q\leq 2P.
\end{equation*} 
Here the implied constant depends on $\nu$ only.
\end{lemma}
\begin{proof}
Take $b_p=1$ for all $p$ in Theorem 1 in \cite{Lau_Wu2}.
\end{proof}

\begin{lemma}\label{2.9}
Let $\nu\in \mathbb{N}$, $2\mid k$ and $N$ be a square free integer.
\par
{\rm (i)}
Define
\begin{equation}\label{2.10}
\mathfrak{P}_\nu^{1}(P,Q)
:= \biggl\{f\in\mathcal{H}_k^*(N): \biggl|\sum_{\substack{P<p\leq Q\\p\nmid N}}\frac{\lambda_f(p^\nu)}{p}\biggr|>\frac{10(\nu+1)}{(\log (kN)) (\log P)}\biggr\}.
\end{equation}
Then we have
\begin{equation*}
|\mathfrak{P}_\nu^{1}(P,Q)|\ll_\nu (kN)^{1-1/(250\nu)},
\end{equation*}
for
\begin{equation}\label{2.4}
\log^{10} (kN)\leq P\leq Q\leq 2P\leq \exp\{\sqrt{\log (kN)}\}.
\end{equation} 
The implied constant depends on $\nu$ at most.
\par
{\rm (ii)}
Let $0 < \varepsilon < 1$ be an arbitrary constant. Define 
\begin{equation}\label{2.11}
\mathfrak{P}_\nu^{2}(P,Q;z)
:= \bigg\{f\in \mathcal{H}_k^*(N): 
\bigg|\sum_{\substack{P<p\leq Q\\p\nmid N}}\frac{\lambda_f(p^\nu)}{p}\bigg|>\left(\frac{96(\nu+1)^2z}{\log_2^2 (kN)P}\right)^{1/2}
\bigg\}.
\end{equation}
Then we have
\begin{equation*}
|\mathfrak{P}_\nu^2(P,Q;z)|
\ll_{\varepsilon,\nu} kN\exp\left\{-c_0(\varepsilon,\nu)\frac{\log (kN)}{\log_2 (kN)}\log \left(\frac{2z}{\varepsilon\log (kN)}\right)\right\},
\end{equation*}
for some positive constant $c_0(\varepsilon,\nu)$ and for 
\begin{equation}\label{2.5}
\varepsilon\log (kN)\leq z\leq P\leq Q\leq 2P\leq \log^{10} (kN).
\end{equation}
The implied constant depends on $\varepsilon$ and $\nu$.
\end{lemma}

\begin{proof}
In Lemma \ref{LargeSieve2}, we choose $j=[\tfrac{\log (kN)}{100\nu \log P}]$ and $j=[\tfrac{\varepsilon\log (kN)}{100\nu \log_2 (kN)}]$ in the proof of (i) and (ii) respectively. 
According to the definition of $\mathfrak{P}_\nu^1(P,Q)$, we have
\begin{equation*}
|\mathfrak{P}_\nu^1(P,Q)|\ll \left(\frac{(\log (kN))\log P}{10(\nu+1)}\right)^{2j}\sum_{f\in \mathcal{H}_k^*(N)}\biggl|\sum_{\substack{P<p\leq Q\\p\nmid N}}\frac{\lambda_f(p^\nu)}{p}\biggr|^{2j}.
\end{equation*}
Then according to the large sieve inequality in Lemma \ref{LargeSieve2} and \eqref{2.4}, we obtain
\begin{equation*}
\begin{aligned}
|\mathfrak{P}_\nu^1(P,Q)|
& \ll \left(\frac{\log (kN) \log P}{10(\nu+1)}\right)^{2j}\left[k\varphi(N)\left(\frac{96(\nu+1)^2j}{P\log P}\right)^j
+(kN)^{10/11}\left(\frac{10Q^{\nu/10}}{\log P}\right)^{2j} \right]
\\
& \ll kN\left(\left(\frac{j(\log P)\log^2(kN)}{P}\right)^j+\frac{Q^{2\nu j}}{(kN)^{1/11}}\right)
\\\noalign{\vskip 1mm}
& \ll (kN)^{1-1/250\nu}.
\end{aligned}
\end{equation*}

Similarly, we have
\begin{equation*}
\begin{aligned}
|\mathfrak{P}_\nu^2(P,Q;z)|
& \ll \left(\frac{P\log_2^2 (kN)}{96(\nu+1)^2z}\right)^j\left[k\varphi(N)\left(\frac{96(\nu+1)^2j}{P\log P}\right)^j
+(kN)^{10/11}\left(\frac{10Q^{\nu/10}}{\log P}\right)^{2j} \right]
\\
&\ll (kN)\left(\left(\frac{j \log_2 (kN)}{z}\right)^j+\frac{Q^{2\nu j}}{(kN)^{1/11}}\right)
\\
&\ll (kN)\exp \left\{-\frac{\varepsilon \log (kN)}{101 \nu \log_2 (kN)} \log\left(\frac{2z}{\varepsilon\log(kN)}\right)\right\},
\end{aligned}
\end{equation*}
for $\log P \geq \tfrac{1}{2}\log_2 (kN)$ and $z\geq \log_2^2 (kN)$.
\end{proof}

\subsection{Proof of Theorem \ref{thm2}(i)}\

\vskip 1mm

In order to prove Theorem \ref{thm2}(i), we need a variant of Lemma \ref{lem2.6}.
\begin{lemma}\label{lem6.3}
Let $1\le m\le 4$, $2\mid k$ and $N$ be a square free integer. Let $0<\varepsilon<1$. Then for $\varepsilon\log (kN)\leq z\leq \log^{10} (kN)$, there exists a constant $c_0=c_0(\varepsilon)$, such that
$$
L(1,{\rm sym}^m f)
=\prod_{\substack{p\leq z\\p\mid N}} \bigg(1-\frac{\lambda_f(p)^m}{p}\bigg)^{-1}
\prod_{\substack{p\leq z\\p\nmid N}}\prod_{0\leq j\leq m} \bigg(1-\frac{\alpha_f(p)^{m-2j}}{p}\bigg)^{-1}
\bigg\{1+O\bigg(\frac{1}{\log_2 (kN)}\bigg)\bigg\}
$$
for all but $O_\varepsilon\big((kN)^{1-c_0(\log [2z/(\varepsilon \log (kN))])/\log_2 (kN)}\big)$ new forms $f\in \mathcal{H}_k^*(N)$. The implied constant is absolute. 
\end{lemma}
\begin{proof}
Let
 \[x=\exp{\sqrt{\tfrac{\log(kN)}{7(m+1)}}},\qquad y_1=\log^{10} (kN), \qquad  y_2=\varepsilon\log (kN).\]
Cut the summations in \eqref{2.8} into three parts:
$p\leq z$
or 
$z<p\leq y_1$
or  
$y_1<p \leq x$.
In view of \eqref{01}, the contribution of the last part, we denote by
\begin{align*}
I_3
& = \sum_{\substack{y_1\leq p\leq x\\p\nmid N}}\sum_{0\leq j\leq m}\log \left(1-\frac{\alpha_f(p)^{m-2j}}{p}\right)^{-1}
+ \sum_{\substack{y_1\leq p\leq x\\p\mid N}}\log \left(1-\frac{\alpha_f(p)^m}{p}\right)^{-1}
\\
& = \sum_{\substack{y_1\leq p\leq x\\p\nmid N}} \frac{\lambda_f(p^m)}{p} 
+ \sum_{\substack{y_1\leq p\leq x\\p\mid N}} \frac{\lambda_f(p^m)}{p}
+ O\big(y_1^{-1}\big)
=:  I_{31} + I_{32} + O\big(y_1^{-1}\big).                     
\end{align*}
For $I_{31}$, use the dyadic method, then we can write
\begin{equation*}
I_{31}=\sum_{1\leq \ell \leq\tfrac{\log (x/y_1)}{\log 2}}\sum_{\substack{2^{\ell-1}y_1< p\leq 2^{\ell}y_1\\p\nmid N}} \lambda_f(p^m)p^{-1}.
\end{equation*}
Define
\[\mathfrak{P}_m^0:=\mathcal{H}^-_k(N; \eta, m)\cup \bigcup_{\ell} \mathfrak{P}_m^1(2^{\ell-1}y_1, 2^{\ell}y_1) 
\qquad \text{for} \qquad \ell\leq \tfrac{\log (x/y_1)}{\log 2},\]
where $\mathfrak{P}_\nu^1(P,Q)$ is defined as in \eqref{2.10}.
Then Lemma \ref{2.9} implies that
\begin{equation*}
|\mathfrak{P}_m^0|\ll (kN)^{65\eta}+\sum_{\ell}|\mathfrak{P}^1_m(2^{\ell-1}y_1,2^\ell y_1)|
\ll (kN)^{65\eta}+(kN)^{1-1/(250m)}\sqrt{\log (kN)}.
\end{equation*}
So for $\eta\in (0,\tfrac{1}{100}]$,
\begin{equation*}
|\mathfrak{P}_m^0|\ll (kN)^{1-1/(250m)+\varepsilon}.
\end{equation*}
Then if $f\in \mathcal{H}_k^*(N)\setminus \mathfrak{P}_m^0$, according to the definition of $\mathfrak{P}_m^1(P,Q)$, we have
\begin{equation*}
\begin{aligned}
I_{31}
& \ll \sum_{1\leq \ell \leq\tfrac{\log (x/y_1)}{\log 2}}\biggl|\sum_{\substack{2^{\ell-1}y_1< p\leq 2^{\ell}y_1\\p\nmid N}}\frac{\lambda_f(p^m)}{p}\biggr|
\\
& \ll \sum_{1\leq \ell \leq\tfrac{\log (x/y_1)}{\log 2}}\frac{10(m+1)}{(\log (kN))\log (2^{\ell-1}y_1)} 
\ll \frac{\log_2 (kN)}{\log (kN)}.
\end{aligned}
\end{equation*}
We can estimate $I_{32}$ directly by
\begin{equation*}
I_{32}
=\sum_{\substack{y_1\leq p\leq x\\p\mid N}} \lambda_f(p^m)p^{-1}
\ll \sum_{\substack{y_1\leq p\leq x\\p\mid N}} p^{-3/2}
\ll \log^{-5}(kN),
\end{equation*}
according to \eqref{eq 1}.
So we get
\begin{equation}\label{3.1}
I_3\ll \tfrac{\log_2 (kN)}{\log (kN)}.
\end{equation}

We denote by $I_2$ the contribution of $z\leq p\leq y_1$. 
As before, we can write
$$
I_2 
= \sum_{\substack{z\leq p\leq y_1\\p\nmid N}} \frac{\lambda_f(p^m)}{p} 
+ \sum_{\substack{z\leq p\leq y_1\\p\mid N}}\frac{\lambda_f(p^m)}{p}
+ O\big(z^{-1}\big)
=: I_{21} + I_{22} + O\big(z^{-1}\big). 
$$
For $I_{21}$, use the dyadic method, then we can write
\begin{equation*}
I_{21}=\sum_{1\leq \ell \leq\tfrac{\log (y_1/z)}{\log 2}}\sum_{\substack{2^{\ell-1}z< p\leq 2^{\ell}z\\p\nmid N}} \lambda_f(p^m)p^{-1}.
\end{equation*} 
Define
\[\mathfrak{P}_m^1(z)=\mathfrak{P}_m^0 \cup \bigcup_\ell \mathfrak{P}_m^2(2^{\ell-1}z,2^\ell z;z)\qquad \text{for} \qquad \ell\leq \tfrac{\log (y_1/z)}{\log 2}\]
where $\mathfrak{P}_m^{2}(P,Q;z)$ is defined as in \eqref{2.11}.
Then Lemma \ref{2.9} implies that
\begin{equation*}
\begin{aligned}
|\mathfrak{P}_m^1(z)|&\ll_\varepsilon (kN)^{1-1/(251m)}+\log_2 (kN)kN\exp\{-c_0(\varepsilon,m)\tfrac{\log (kN)}{\log_2(kN)}\log(\tfrac{2z}{\varepsilon\log (kN)})\}\\
            &\ll_\varepsilon (kN)^{1-c_1\{\log(2z/\varepsilon\log(kN))\}/\log_2(kN)},
\end{aligned}
\end{equation*}
where $c_1=c_1(\varepsilon,m)$ is a positive constant depends on $\varepsilon$ and $m$.
Then if $f\in \mathcal{H}^*_k(N)\setminus \mathfrak{P}_m^1(z)$, 
according to the definition of $\mathfrak{P}_\nu^2(P,Q;z)$, we have
\begin{align*}
I_{21}
& \ll \sum_{\ell} \frac{\sqrt{z}}{\log_2 (kN)\cdot \sqrt{ 2^{\ell-1}z}}\ll \frac{1}{\log_2(kN)},
\\
I_{22}
& =\sum_{\substack{z\leq p\leq y_1\\p\mid N}} \frac{\lambda_f(p^\nu)}{p}
\ll \frac{1}{\sqrt{z}}
\ll \frac{1}{\log_2(kN)}.
\end{align*}
Then we have $I_2\ll \log_2^{-1}(kN)$.
Together with \eqref{3.1}, we obtain
\begin{align*}
\log L(1,{\rm sym}^m f)
&=\sum_{\substack{p\leq z\\ p\nmid N}}\sum_{0\leq j\leq m}\log\bigg(1-\frac{\alpha_f(p)^{m-2j}}{p}\bigg)^{-1}
\\
& \quad
+\sum_{\substack{p\leq z\\ p\mid N}}\log\bigg(1-\frac{\lambda_f(p)^{m}}{p}\bigg)^{-1}
+O\bigg(\frac{1}{\log_2(kN)}\bigg),
\end{align*}
for $f\in \mathcal{H}^*_k(N)\setminus \mathfrak{P}_m^1(z)$. 
It implies the required result immediately.
\end{proof}

Now we are ready to prove Theorem \ref{thm2}(i).
According to Lemma \ref{lem6.3}, there are constants $c_0=c_0(\varepsilon)$, $k_0=k_0(\varepsilon)$ and $N_0=N_0(\varepsilon)$, such that for $k\geq k_0$, $N\geq N_0$ and $\varepsilon \log (kN)\leq z\leq \log^{10} (kN)$, we can find a subset $\mathfrak{P}_{k,N}^*(z)\subset \mathcal{H}_k ^*(N)$, with
\[|\mathfrak{P}_{k,N}^*(z)|
\ll kN \exp\left\{-c_0\log\left(\frac{2z}{\varepsilon \log (kN)}\right)\frac{\log (kN)}{\log_2(kN)}\right\},\]
such that for all $f\in \mathcal{H}^*_k(N)\setminus \mathfrak{P}_{k,N}^*(z)$, 
the formula of Lemma \ref{lem6.3} holds.

For these $f\in \mathcal{H}^*_k(N)\setminus \mathfrak{P}_{k,N}^*(z)$, 
when $N\in \mathbb{N}_k(\Xi)$, we have
\begin{equation*}
\begin{aligned}
L(1,{\rm sym}^m f)
& \leq \bigg\{1+O\left(\frac{1}{\log_2 (kN)}\right)\bigg\}
\prod_{p\leq z, \, p\mid N} \left(1-\frac{\lambda_f(p)^m}{p}\right)^{-1}
\prod_{p\leq z, \, p\nmid N} \left(1-\frac{1}{p}\right)^{-(m+1)}
\\
& \leq \bigg\{1+O\left(\frac{1}{\log_2 (kN)}\right)\bigg\}
\bigg\{1+O\left(\frac{\omega(N)}{P^-(N)}\right)\bigg\} ({\rm e}^\gamma\log z)^{m+1}
\\
& \leq (B_m^+(\log z+C_0))^{A_m^+},
\end{aligned}
\end{equation*}
where $\omega(N)\ll \log (2N)/\log_2(3N)$ is the number of prime factors of $N$.

Similarly, we have
\begin{align*}
L&(1,{\rm sym}^m f)
 = \big\{1+O\big(\log_2^{-1}(kN)\big)\big\}
\prod_{\substack{p\leq z\\p\nmid N}}\prod_{0\leq j\leq m} \bigg(1-\frac{\alpha_f(p)^{m-2j}}{p}\bigg)^{-1}
\\
& \geq \big\{1+O\big(\frac{1}{\log_2(kN)}\big)\big\}\exp\Big\{-\frac{1}{A_m^-}\sum_{p\leq z} \big(-\log D(p^{-1},{\rm sym}^m[g(\theta_{m,p}^{-})])-\frac{A_m^-}{p}\big)
-\sum_{p\leq z} \frac{A_m^-}{p}\Big\}  
\\
& \geq (B_m^-(\log z+C_0))^{-A_m^-}.
\end{align*}

Then we can complete the proof of Theorem \ref{thm2}(i) by taking $z=\exp\{\log_2 (kN)+\phi-C_0\}$.

\subsection{Proof of Theorem \ref{thm2}(ii)}\

\vskip 1mm

This is an immediate consequence of lower bound part of Corollary \ref{Cor}.

\vskip 8mm

\section{Proof of Theorem \ref{thm3}}

In this section, we will refine the argument of Lamzouri \cite{Lamzouri} and apply a little more tricks from \cite{Liu_Royer_Wu} to proof Theorem \ref{thm3}.
We only consider the case of sign $-$,
and the other case can be treated in the same way.
First of all, we need to improve the estimate of \eqref{e 18} by giving more precise error term. Then the following lemma is an analogue of \cite[Lemma 1.1]{Lamzouri}.

\begin{lemma}\label{lem7.1}
Let
$$
h_m^{-}(x) := \begin{cases}
\displaystyle\log \left(\frac{2}{\pi}\int_{0}^{\pi}
\exp\left(-\frac{{\rm tr}({\rm sym}^m[g(\theta)])}{m+1}x \right)\sin^2 \theta\,\dd{\theta}\right)        &  \text{si $x< 1$,}
\\\noalign{\vskip 2mm}
\displaystyle\log\bigg(\frac{2}{\pi}\int_{0}^{\pi}
\exp\left(-\frac{{\rm tr}({\rm sym}^m[g(\theta)])}{m+1}x \right)\sin^2 \theta\,\dd{\theta}\bigg)-\frac{A_m^-}{m+1}x        &  \text{si $x\ge 1$.}
\end{cases}
$$ 
Then we have
\begin{equation*}
   h_m^{-}(x)=\left\{
    \begin{aligned}[l]
      &O(x^2)        &        &\quad (x< 1),
      \\
      &O(\log x+1)        &        &\quad (x\ge 1).
    \end{aligned} 
    \right.
\end{equation*}
\end{lemma}

\begin{proof}
The proof is almost the same with \cite[Lemma 1.1]{Lamzouri} in view of the following equation
$$
{\rm tr}({\rm sym}^m[g(\theta)])=-A_m^-+c_m(\theta-\theta_{m,p}^-)^2+O\big((\theta-\theta_{m,p}^-)^3\big)
$$
for some positive constant $c_m$ and
$\theta_{m,p}^-$ is defined by \eqref{defthetamp}.
\end{proof}

The next lemma is an improvement of \eqref{e 18}, 
which is needed in the proof of Theorem \ref{thm3}.

\begin{lemma}\label{lem7.2}
Let $m$ be a positive integer and $M^{\pm r}_{{\rm sym}^m}$ be defined as in \eqref{defMsymmz}. 
Then we have
$$
\log M^{\pm r}_{{\rm sym}^m}
= A^{\pm}_{m}r\log (B^{\pm}_m\log(A_m^{\pm}r))
+ \frac{A_m^{\pm}r}{\log(A_m^{\pm}r)} 
\left\{\mathscr{A}_m^{\pm}-1 + \frac{\mathscr{B}_m^{\pm}}{\log(A_m^{\pm}r)} + O\left(\frac{1}{(\log r)^2}\right)\right\}
$$
for $r\to\infty$,
where $A^{\pm}_m$ and $B^{\pm}_m$ are positive constants defined as in \eqref{defAmpmBmpm} above and $\mathscr{A}_m^+$ and $\mathscr{B}_m^+$ are given by \eqref{DefAm} and \eqref{DefBm} and
\begin{equation*}
\mathscr{A}_m^- := \mathscr{D}_m^- + \log\bigg(\frac{m+1}{A_m^-}\bigg),
\qquad 
\mathscr{B}_m^-=\mathscr{K}_m^--\frac{1}{2}\log^2\bigg(\frac{m+1}{A_m^-}\bigg).
\end{equation*}
Here, $\mathscr{D}_m^-$, $\mathscr{K}_m^-$ are defined by \eqref{DefDK}.
\end{lemma}

\begin{proof}
For $+r$, a little variant in the proof of \cite[Proposition 1.2]{Lamzouri}  gives
\begin{align}
\mathscr{A}_m^+
& := 1+\int_0^1 \frac{h_m^+(t)}{t^2}\,\dd{t}+\int_1^\infty \frac{h_m^+(t)-t}{t^2}\,\dd{t},
\label{DefAm}
\\ 
\mathscr{B}_m^+
& := \int_0^1 \frac{h_m^+(t)}{t^2}\log t\,\dd{t}+\int_1^\infty \frac{h_m^+(t)-t}{t^2}\log t\,\dd{t},
\label{DefBm}
\end{align}
where
$$
h_m^+(t) 
:= \log\bigg(\frac{2}{\pi}\int_0^\pi\exp\bigg(\frac{t}{m+1}\sum_{j=0}^m\cos(\theta(m-2j))\bigg)\sin^2\theta\,\dd{\theta}\bigg).
$$

For $-r$, we recall that
$$
M^{-r}_{{\rm sym}^m}
= \prod_{p}\frac{2}{\pi}\int_0^{\pi}D\big(p^{-1},{\rm sym}^m[g(\theta)]\big)^{-r}\sin^2 \theta\,\dd{\theta}
=: \prod_{p}\mathscr{E}_p^{-r}.$$

For $p\leqslant \sqrt{(m+1)r}$, we write for $|\theta-\theta_{m,p}^-|<\delta$, 
\begin{align*}
D\big(p^{-1},{\rm sym}^m[g(\theta)]\big)
& = D(p^{-1},{\rm sym}^m[g(\theta_{m,p}^-)])
\\
& \quad
+ \tfrac{1}{2}D''(p^{-1},{\rm sym}^m[g(\theta_{m,p}^-)])(\theta-\theta_{m,p}^-)^2
+ O\big((\theta-\theta_{m,p}^-)^3\big),
\end{align*}
where $\delta$ is a small parameter chosen later. Then
\begin{equation*}
\begin{aligned}
& \mathscr{E}_p^{-r} D(p^{-1},{\rm sym}^m[g(\theta_{m,p}^-)])^r
\\
& \geqslant \frac{2}{\pi}
\int_{\theta_{m,p}^--\delta}^{\theta_{m,p}^-+\delta}
\bigg(\frac{D\big(p^{-1},{\rm sym}^m[g(\theta)]\big)}{D(p^{-1},{\rm sym}^m[g(\theta_{m,p}^-)])}\bigg)^{-r}\sin^2 \theta\,\dd{\theta}
\\
& = \frac{2}{\pi}\int_{\theta_{m,p}^--\delta}^{\theta_{m,p}^-+\delta}
\left\{1
+\frac{D''(p^{-1},{\rm sym}^m[g(\theta_{m,p}^-)])}{2D(p^{-1},{\rm sym}^m[g(\theta_{m,p}^-)])}(\theta-\theta_{m,p}^-)^2
+O\big((\theta-\theta_{m,p}^-)^3\big)\right\}^{-r}\sin^2 \theta\,\dd{\theta}.
\end{aligned}
\end{equation*}
Since $D\big(p^{-1},{\rm sym}^m[g(\theta)]\big)\geqslant \left(1+p^{-1}\right)^{-(m+1)}$ and
\begin{equation*}
\frac{\dd^2}{\dd\theta^2}\log D\big(p^{-1},{\rm sym}^m[g(\theta)]\big)\ll_m\frac{1}{p}, 
\qquad 
\frac{\dd}{\dd\theta}\log D(p^{-1},{\rm sym}^m[g(\theta_{m,p}^-)])=0, 
\end{equation*}
we can write
\begin{equation*}
\mathscr{E}_p^{-r}D(p^{-1},{\rm sym}^m[g(\theta_{m,p}^-)])^r
\geqslant \left\{1+O\bigg(\frac{\delta^2}{p}+\delta^3\bigg)\right\}^{-r}
\frac{2}{\pi}\int_{\theta_{m,p}^--\delta}^{\theta_{m,p}^-+\delta}\sin^2 \theta\,\dd{\theta}.
\end{equation*}
Since
$$
\bigg|\log\bigg(\frac{2}{\pi}\int_{\theta_{m,p}^--\delta}^{\theta_{m,p}^-+\delta}\sin^2 \theta\,\dd{\theta}\bigg)\bigg|\ll -\log \delta+1,
$$
we chose $\delta\leqslant p/r^c$ for some sufficiently large constant $c>0$, then we get
\begin{equation}\label{e 6.10}
\log \mathscr{E}_p^{-r}=-r\log D(p^{-1},{\rm sym}^m[g(\theta_{m,p}^-)])+O_m(\log r).
\end{equation}

For $p>\sqrt{(m+1)r}$, we have
\begin{align*}
D\big(p^{-1},{\rm sym}^m[g(\theta)]\big)^{-r}
& =\prod_{j=0}^{m}\left(1-\frac{2\cos(m-2j)\theta}{p}+\frac{1}{p^2}\right)^{r/2}
\\
& = \exp \left\{-\frac{r {\rm tr}({\rm sym}^m[g(\theta)]) }{p}+O_m\bigg(\frac{r}{p^2}\bigg)\right\}
\\
& = \exp \left(-\frac{r {\rm tr}({\rm sym}^m[g(\theta)]) }{p}\right)
\left\{1+O_m\bigg(\frac{r}{p^2}\bigg)\right\}.
\end{align*}
In view of \eqref{e 3.4} and together with \eqref{e 6.10}, we have
\begin{equation}\label{Eq.A}
\begin{aligned}
\log M^{- r}_{{\rm sym}^m}
& = -r\sum_{p\leqslant (m+1)r}\log D(p^{-1},{\rm sym}^m[g(\theta_{m,p}^-)])
\\
& \quad
+ \sum_{p>\sqrt{(m+1)r}} h_m^{-}((m+1)r/p) + O_m(r^{1/2})
= : S_1 + S_2 + O_m(r^{1/2}).
\end{aligned}
\end{equation}
First we evaluate $S_1$.
In view of \eqref{e 3.4} and \eqref{defAmpmBmpm}, we can write
\begin{align*}
S_1
& = r\sum_{p\leqslant (m+1)r} \big(-\log D(p^{-1},{\rm sym}^m[g(\theta_{m,p}^-)])-A_m^{-}p^{-1}\big)
+ r\sum_{p\le (m+1)r} A_m^{-}p^{-1}
\\
& = A_m^{-}r\log(B_m^{-}\log (A_m^-r))+\frac{A_m^-r\log ((m+1)/A_m^-)}{\log (A_m^-r)}\\
\noalign{\vskip 2mm}
&\quad\qquad\qquad\qquad\qquad\qquad
-\frac{A_m^-r(\log (A_m^-/(m+1)))^2}{2(\log (A_m^-r))^2}
+ O\bigg(\frac{r}{(\log r)^3}\bigg).
\end{align*}
Following the method of Lamzouri (see \cite[1.5-1.9]{Lamzouri}) 
with a little more effort to precise the coefficient of the term $1/(\log r)^2$, we can obtain
$$
S_2 = \frac{A_m^-r}{\log (A_m^-r)}
\left\{\mathscr{D}_m^--1+\frac{\mathscr{K}_m^-}{\log (A_m^-r)} + O\bigg(\frac{1}{(\log r)^2}\bigg)\right\},
$$
where
\begin{equation}\label{DefDK}
\mathscr{D}_m^-
:= 1 + \int_0^\infty \frac{h_m^{-}(u)}{u^2}\,\dd{u},
\qquad 
\mathscr{K}_m^-
:= \int_0^\infty \frac{h_m^{-}(u)}{u^2}\log u\,\dd{u}.
\end{equation}
Inserting into \eqref{Eq.A}, we can complete the proof for $-r$.
\end{proof}

\begin{remark}
We write $\mathscr{A}_m^{\pm}-1$ only for the convenience of later use.
\end{remark}

Now we are ready to prove Theorem \ref{thm3}.

For $1\le m\le 4$, We define
\begin{equation*}
\mathfrak{F}_{k,N}^{\pm}(t)
:= \sum_{\substack{f\in \mathcal{H}_k^{*}(N)\\ L(1,{\rm sym}^mf)\gtreqless (B_m^{\pm}t)^{\pm A_m^{\pm}}}}
\omega_f, \quad \mathfrak{F}_{k,N}^{\pm,*}(t)
:= \sum_{\substack{f\in \mathcal{H}_k^{+}(N;\eta,m)\\ L(1,{\rm sym}^mf)\gtreqless (B_m^{\pm}t)^{\pm A_m^{\pm}}}}
\omega_f.
\end{equation*}
In view of \eqref{harmweight} and \eqref{2.12-}, we have
\begin{equation}\label{e 7.6}
\mathfrak{F}_{k,N}^{\pm}(t)=\mathfrak{F}_{k,N}^{\pm,*}(t)+O((kN)^{-4/5}).
\end{equation}

We only consider the case with sign $-$.
First we write
\begin{equation*}
\begin{aligned}
A_m^-r\int_{0}^\infty \mathfrak{F}_{k,N}^{-,*}(t)t^{A_m^-r-1}\dd{t}
& = A_m^-r\int_{0}^\infty t^{A_m^-r-1} 
\sum_{\substack{f\in \mathcal{H}_k^+(N;\eta,m)\\L(1,{\rm sym}^mf)\le (B_m^-t)^{-A_m^-}}}\omega_f \,\dd{t}
\\
& = (B_m^-)^{-A_m^-r}\sum_{f\in \mathcal{H}_k^+(N,\eta,m)}\omega_f L(1,{\rm sym}^mf)^{-r}.
\end{aligned}
\end{equation*}
Together with Proposition \ref{prop4.1}, uniformly for 
$|r|\leq c\log (kN)/\log_2(10kN)\log_3(10kN)$, we get
$$
A_m^-r\int_{0}^\infty \mathfrak{F}_{k,N}^{-,*}(t)t^{A_m^-r-1}\dd{t}
= (B_m^-)^{-A_m^-r}M^r_{{\rm sym}^m}(N) + O\big(\text{e}^{-c'\log(kN)/\log_2(kN)}\big).
$$
Thanks to \eqref{N1} and Lemma \ref{lem7.2}, on can deduce that
\begin{equation}\label{e 6.3}
\begin{aligned}
& A_m^-r\int_{0}^\infty \mathfrak{F}_{k,N}^{-,*}(t)t^{A_m^-r-1}\dd{t}
\\
& = (\log(A_m^-r))^{A_m^-r}
\exp \left(\frac{A_m^-r}{\log (A_m^-r)} \left\{\mathscr{A}_m^--1
+ \frac{\mathscr{B}_m^-}{\log(A_m^-r)} + O\left(\frac{1}{(\log r)^2}\right)\right\}\right).
\end{aligned}
\end{equation}

Let $\varpi$ be a small positive parameter to be chosen later,
$\tau =\log(A_m^-r)+\mathscr{A}^-_m$ and $R=r{\rm e}^\varpi$. 
Then by using \eqref{e 6.3} with $R$ in place of $r$, we have
\begin{align*}
\int_{\tau+\varpi}^\infty \mathfrak{F}_{k,N}^{-,*}(t)t^{A_m^-r-1}\dd{t}
& \leqslant (\tau+\varpi)^{A_m^-(r-R)}\int_{0}^\infty \mathfrak{F}_{k,N}^{-,*}(t)t^{A_m^-R-1}\dd{t}
\\
& = \Upsilon
\exp\left(\frac{A_m^-R}{\log (A_m^-R)}\left\{\mathscr{A}^-_m-1
+ \frac{\mathscr{B}_m^-}{\log A_m^- r} + O\left(\frac{1}{(\log r)^2}\right)\right\}\right),
\end{align*}
where
\begin{align*}
\Upsilon
& = (\tau+\varpi)^{A_m^-(r-R)} (\log(A_m^-R))^{A_m^-R}
= (\log(A_m^-r)+\mathscr{A}^-_m+\varpi)^{A_m^-r} 
\bigg(1-\frac{\mathscr{A}_m^-}{\tau+\varpi}\bigg)^{A_m^-R}
\\
& = (\log(A_m^-r))^{A_m^-r}
\exp\bigg\{\frac{A_m^-r}{\log(A_m^-r)}
\bigg[(\mathscr{A}^-_m+\varpi 
- \text{e}^{\varpi} \mathscr{A}_m^-)
- \frac{{\mathscr{A}^-_m}^2(1-\text{e}^{\varpi})}{2\log (A_m^-r)}\\
& \qquad \qquad \qquad \qquad \qquad 
+ \frac{2\text{e}^{\varpi}\varpi \mathscr{A}_m^--2\varpi \mathscr{A}_m^--\varpi^2}{2\log (A_m^-r)}
+ O\bigg(\frac{1}{(\log r)^2}\bigg)\bigg]\bigg\}. 
\end{align*}
On the other hand, we also have
$$
\frac{A_m^-R}{\log (A_m^-R)}
= \frac{A_m^-r\text{e}^{\varpi}}{\log (A_m^-r)+\varpi}
= \frac{A_m^-r\text{e}^{\varpi}}{\log (A_m^-r)}\bigg\{1-\frac{\varpi}{\log(A_m^-r)}
+O\bigg(\frac{\varpi^2}{(\log r)^2}\bigg)\bigg\}.
$$
Inserting these in the preceding inequality and taking $\varpi=C/\log(A_m^-r)$ for some constant $C$ large enough, 
we find that
\begin{align*}
& A_m^-r\int_{\tau+\varpi}^\infty \mathfrak{F}_{k,N}^{-,*}(t)t^{A_m^-r-1}\dd{t}
\\
& \le (\log(A_m^-r))^{A_m^-r}
\exp\bigg\{\frac{A_m^-r}{\log(A_m^-r)}
\bigg[\mathscr{A}^-_m+\varpi 
- \text{e}^{\varpi}
+ \frac{\mathscr{B}_m^-}{\log r} + O\left(\frac{1}{(\log r)^2}\right)\bigg]\bigg\},
\end{align*}
which and \eqref{e 6.3} imply that (for large constant $C$)
\begin{equation}\label{e 6.1}
A_m^-r\int_{\tau+\varpi}^\infty \mathfrak{F}_{k,N}^{-,*}(t)t^{A_m^-r-1}\dd{t}
\leqslant A_m^-r\int_{0}^\infty \mathfrak{F}_{k,N}^{-,*}(t)t^{A_m^-r-1}\,\dd{t}
\exp\left(-\frac{r}{(\log r)^3}\right).
\end{equation}
Similarly, we can get
\begin{equation}\label{e 6.2}
A_m^-r\int^{\tau-\varpi}_0 \mathfrak{F}_{k,N}^{-,*}(t)t^{A_m^-r-1}\dd{t}
\leqslant A_m^-r\int_{0}^\infty \mathfrak{F}_{k,N}^{-,*}(t)t^{A_m^-r-1}\,\dd{t}
\exp\left(-\frac{r}{(\log r)^3}\right).
\end{equation}
Thus, one can deduce from \eqref{e 6.3}-\eqref{e 6.2}
\begin{equation}\label{e 6.4}
\begin{aligned}
& A_m^-r \int_{\tau-\varpi}^{\tau+\varpi} \mathfrak{F}_{k,N}^{-,*}(t)t^{A_m^-r-1}\dd{t}
\\
& = (\log A_m^-r)^{A_m^-r}
\exp\left(\frac{A_m^-r}{\log (A_m^-r)}\left\{\mathscr{A}_m^--1+O\left(\frac{1}{\log r}\right)\right\}\right).
\end{aligned}
\end{equation}

Since $\mathfrak{F}_{k,N}^{-,*}(t)$ is non-increasing, we have
\begin{equation*}
\mathfrak{F}_{k,N}^{-,*}(\tau+\varpi)\tau^{A_m^-r}\exp\big\{O\big(\varpi r/\tau\big)\big\} 
\leqslant A_m^-r \int_{\tau-\varpi}^{\tau+\varpi} \mathfrak{F}_{k,N}^{-,*}(t)t^{A_m^-r-1}\dd{t}.
\end{equation*}
On the other hand, we have
$$
A_m^-r \int_{\tau-\varpi}^{\tau+\varpi} \mathfrak{F}_{k,N}^{-,*}(t)t^{A_m^-r-1}\dd{t} 
\leqslant \mathfrak{F}_{k,N}^{-,*}(\tau-\varpi)\tau^{A_m^-r}\exp\big\{O\big(\varpi r/\tau\big)\big\}.
$$
Considering these two inequalities together with \eqref{e 6.4}, one obtains
\begin{equation}\label{initial}
\mathfrak{F}_{k,N}^{-,*}(\tau+\varpi)
\leqslant \exp\bigg(-\frac{{\rm e}^{\tau-\mathscr{A}_m^-}}{\tau}\{1+O(\varpi)\}\bigg) 
\leqslant \mathfrak{F}^{-,*}_{k,N}(\tau-\varpi).
\end{equation}
for $kN\to\infty$ with $2\mid k$ and $N\in \mathbb{N}_k(\Xi)$ and 
$\tau\le \log_2(kN)-\log_3(kN)-\log_4(kN)-\tfrac{1}{2}c_{11}$,
where $c_{11} := 2\big(-\log(cA_m^{-})-\mathscr{A}_m^{-})$ is a positive constant.

For any $t\leqslant \log_2 (kN)-\log_3 (kN)-\log_4 (kN)-c_{11}$, 
we apply \eqref{initial} with $\tau_1=t-\varpi$ and $\tau_2=t+\varpi$ to write
\begin{align*}
\mathfrak{F}_{k,N}^{-,*}(t)
= \mathfrak{F}_{k,N}^{-,*}(\tau_1+\varpi)
\leqslant \exp\bigg(-\frac{{\rm e}^{t-\varpi-\mathscr{A}_m^-}}{t-\varpi}\{1+O(\varpi)\}\bigg)
= \exp\bigg(-\frac{{\rm e}^{t-\mathscr{A}_m^-}}{t}\{1+O(\varpi)\}\bigg),
\\
\mathfrak{F}_{k,N}^{-,*}(t)
= \mathfrak{F}_{k,N}^{-,*}(\tau_2-\varpi)
\geqslant \exp\bigg(-\frac{{\rm e}^{t+\varpi-\mathscr{A}_m^-}}{t+\varpi}\{1+O(\varpi)\}\bigg)
= \exp\bigg(-\frac{{\rm e}^{t-\mathscr{A}_m^-}}{t}\{1+O(\varpi)\}\bigg).
\end{align*}
Together with \eqref{e 7.6} and the following equality
$$
\sum_{f\in \mathcal{H}^*_{k}(N)}\omega_f = 1+O\big(k^{-5/6}N^{-1+\varepsilon}\big),
$$
the estimate for $\mathscr{F}_{k,N}^{-}(t, {\rm sym}^m)$ follows.

\section*{Acknowledgement}
This work is part of the author's PhD thesis in IECL. Thanks are due to Jie Wu for his valuable advise.


\end{document}